\documentclass[aap]{imsart_new}

\RequirePackage{amsthm,amsmath,amsfonts,amssymb}
\RequirePackage[numbers,sort&compress]{natbib}
\RequirePackage[colorlinks,citecolor=blue,urlcolor=blue]{hyperref}
\RequirePackage{graphicx}

\usepackage{dsfont} 
\usepackage{mathtools} 
\mathtoolsset{showonlyrefs}
\usepackage{subfigure} 
\usepackage{todonotes} 
\usepackage{tabularx,booktabs}

\startlocaldefs
\theoremstyle{plain}
\newtheorem{theorem}{Theorem}
\newtheorem{proposition}[theorem]{Proposition}
\newtheorem{lemma}[theorem]{Lemma}

\theoremstyle{remark}

\newtheorem{assumption}[theorem]{Assumption}

\newtheorem{conjecture}[theorem]{Conjecture}

\newcommand{\tinyspace}{\hspace{0.06em}}

\DeclareMathOperator*{\argmax}{arg\tinyspace{}max}

\newcommand{\bigO}{\ensuremath{\mathrm{O}}}
\newcommand{\bigOp}[1][\P]{\ensuremath{\bigO_{\scriptscriptstyle{}#1}}}
\newcommand{\smallO}{\ensuremath{\mathrm{o}}}
\newcommand{\smallOp}[1][\P]{\ensuremath{\smallO_{\scriptscriptstyle{}#1}}}

\newcommand{\bigOmega}{\ensuremath{\mathrm{\Omega}}}
\newcommand{\bigOmegap}[1][\P]{\ensuremath{\bigOmega_{\scriptscriptstyle{}#1}}}

\newcommand{\sss}{\scriptscriptstyle}

\DeclareRobustCommand{\HOF}[3]{#2} 

\endlocaldefs

\begin{document}
\numberwithin{theorem}{section}

\renewcommand{\P}{\ensuremath{\mathbb{P}}}
\newcommand{\E}{\ensuremath{\mathbb{E}}}
\newcommand{\V}{\ensuremath{\mathbb{E}}}
\newcommand{\R}{\ensuremath{\mathbb{R}}}
\newcommand{\N}{\ensuremath{\mathbb{N}}}
\newcommand{\e}{\ensuremath{\mathrm{e}}}
\renewcommand{\epsilon}{\varepsilon}


\begin{frontmatter}
\title{Detecting a late changepoint\\in the preferential attachment model}
\runtitle{Detecting a late changepoint in the preferential attachment model}

\begin{aug}
\author[A]{\fnms{Gianmarco}~\snm{Bet}\ead[label=e1]{gianmarco.bet@unifi.it}\orcid{0000-0001-8431-0636}},
\author[B]{\fnms{Kay}~\snm{Bogerd}\ead[label=e2]{kaybogerd@hotmail.com}}
\author[B]{\fnms{Rui~M.}~\snm{Castro}\ead[label=e3]{rmcastro@tue.nl}\orcid{0000-0003-4398-0718}}
\author[B]{\fnms{Remco}~\snm{van~der~Hofstad}\ead[label=e4]{r.w.v.d.hofstad@tue.nl}\orcid{0000-0003-1331-9697}}
\address[A]{Department of Mathematics and Computer Science, University of Florence\printead[presep={,\ }]{e1}}
\address[B]{Department of Mathematics and Computer Science, Eindhoven University of Technology\printead[presep={,\ }]{e2,e3,e4}}
\end{aug}

\begin{abstract}
Motivated by the problem of detecting a change in the evolution of a network, we consider the preferential attachment random graph model with a \emph{time-dependent} attachment function. Our goal is to detect whether the attachment mechanism changed over time, based on a single snapshot of the network and without directly observable information about the dynamics. We cast this question as a hypothesis testing problem, where the null hypothesis is a preferential attachment model with a constant affine attachment parameter $\delta_0$, and the alternative hypothesis is a preferential attachment model where the affine attachment parameter changes from $\delta_0$ to $\delta_1$ at an unknown changepoint time $\tau_n$. For our analysis we focus on the regime where $\delta_0$ and $\delta_1$ are fixed, and the changepoint occurs close to the observation time of the network (i.e., $\tau_n = n - c n^\gamma$ with $c>0$ and $\gamma \in (0, 1)$). This corresponds to the relevant scenario where we aim to detect the changepoint shortly after it has happened.

We present two tests based on the number of vertices with minimal degree, and show that these are asymptotically powerful when $ \tfrac{1}{2} < \gamma < 1$. We conjecture that there is no powerful test based on the final network snapshot when $\gamma<\tfrac{1}{2}$. The first test we propose requires knowledge of $\delta_0$. The second test is significantly more involved, and does not require the knowledge of $\delta_0$ while still achieving the same performance guarantees. Furthermore, we prove that the test statistics for both tests are asymptotically normal, allowing for accurate calibration of the tests. This is demonstrated by numerical experiments, that also illustrate the finite sample test properties.
\end{abstract}



\end{frontmatter}
 \tableofcontents %


\section{Introduction}
\label{sec:changepoint_detect_introduction}
One of the most celebrated successes of complex network theory has been the recognition that simple \emph{dynamical} random graph models with local connection rules are able to successfully explain important macroscopic features observed in real-world networks. The preferential attachment model and its generalizations are perhaps the most successful of such models. Barab\'{a}si and Albert \cite{Barabasi1999} proposed this model to explain the occurrence of power-law degree sequences, which are often observed in real-world networks such as the world wide web \cite{Broder2000,Adamic2000} or internet \cite{Faloutsos1999}, biological networks \cite{Jeong2000,Farkas2003,Middendorf2005}, collaboration networks of movie actors \cite{Albert2000,Gao2011}, and citation networks \cite{Redner1998,Newman2001,Barabasi2002,Wang2008}.
Furthermore, the typical distance between vertices in the preferential attachment model is small \cite{DomHofHoo10} (see also \cite[Chapter 8]{VanderHofstad2020} and references therein). This is called the \emph{small-world} phenomenon \cite{Watts1998,Watts1999}.

The preferential attachment model considers the entire evolution of a network by adding vertices one by one using a simple \emph{preferential attachment} rule. Informally, as new vertices are added to the graph, they are more likely to attach to vertices that already have a large degree, hence further increasing the degree of these vertices. This formalism essentially creates a paradigm where ``the rich get richer'', which is often invoked to explain the wide-spread inequality in socio-economic contexts \cite{perk:2014}. Accordingly, the degree of the oldest vertices grows as new vertices attach to the graph. On the other hand, the degree of the last few vertices to join is typically quite small. Since its introduction in \cite{Barabasi1999}, the preferential attachment model has received a tremendous amount of attention thanks to its early explanatory successes. The structural properties of the model are investigated formally in \cite{Bollobas2001a,Bollobas2004}, see also \cite{VanderHofstad2017,VanderHofstad2020} for a detailed overview on this model and many of its properties. 

In our work we are interested in situations where the growth dynamics of the network do not remain constant over time, but have a change at some point. This captures a situation where a major event could cause a change in the subsequent evolution of the network. To model this, we consider a time-inhomogeneous affine preferential attachment model, where a new vertex $v_t$ that enters the graph at time $1 \leq t \leq n$ connects to a pre-existing vertex with degree $k$ with probability proportional to $f(k) = k + \delta(t)$. We consider the hypothesis testing problem where $\delta(t) = \delta_0$ remains constant under the null hypothesis, whereas under the alternative the affine attachment parameter $\delta(t)$ changes from $\delta_0$ to $\delta_1$ at an unknown moment $\tau_n$, called a \emph{changepoint}. For our work we are particularly interested in scenarios where the change occurs very late, and affects only a very small part of the graph. Specifically, in the regime we are interested in, the changepoint has the form $\tau_n = n - c n^\gamma$ with $c > 0$ and $\gamma \in (0, 1)$, as explained in Section~\ref{sec:model}. From a practical standpoint this is relevant when one wants to detect the change as quickly as possible.

\subsection{Related work}
Our work nicely complements earlier results \cite{Bhamidi2018,Banerjee2018} that focus on the detection of a changepoint in the setting of preferential attachment trees, where every vertex that enters the graph connects to $m = 1$ other vertices. There are also some differences. First, our results consider the more general case of preferential attachment graphs, where vertices may enter the graph with $m \geq 1$ edges. The other difference is that we focus on a \textit{late} changepoint $\tau_n = n - c n^\gamma$, whereas \cite{Bhamidi2018,Banerjee2018} focus on a changepoint that happens at a linear time $\bigO(n)$ or even $\smallO(n)$. Thus, in our setting a much smaller number of vertices enter the graph after the changepoint, making it harder to detect. The authors of \cite{Cirkovic2022} propose a likelihood-ratio testing procedure to detect a changepoint in a preferential attachment tree and the associated changepoint estimator. Crucially the methods in \cite{Cirkovic2022} rely on the knowledge of the entire network evolution. This is not the case for our test, which only requires a snapshot of the network at the final time. The authors of \cite{Cirkovic2022} extend their test to detect multiple changepoints by applying two general techniques (namely, Screening and Ranking, and Binary Segmentation) to decompose the multiple-changepoints problem into a sequence of single changepoint problems. This work, however, is still in the scenario where the changepoint occurs at a linear time, in stark contrast with the regime we investigate. 

Although different from this work, there has been much interest in understanding and detecting the effect of an initial seed graph on the evolution of the preferential attachment tree \cite{Bubeck2015,Curien2015,Bubeck2017,Bubeck2017a,Marchand2020}. Here one starts with a given initial graph at time $t = 1$ and then grows the remaining tree according to a preferential (or uniform) attachment. The goal is to estimate the initial seed graph based on an observation of the fully developed graph at a much later time. Finally, changepoint detection and related inference questions have also received much attention in the setting of dynamic stochastic block models \cite{Wang2014,Wang2018,Pensky2019,Zhao2019,Bhattacharjee2020}. In those works the aim is primarily to understand the evolution of the network's community structure.

\section{Model}\label{sec:model}
We formalize the problem of detecting a changepoint in a dynamical network as a hypothesis testing problem on random graphs. We first explain the model that we use in general, and then define concrete versions of this model for the null and alternative hypothesis.
This model has parameters $m\in\N$ and $\delta:\N\to(-m,\infty)$ and produces a sequence of undirected graphs without loops. Let $G_n = (V_n, E_n)$ be an undirected graph, where $V_n = \{v_0, \ldots, v_n\}$ denotes the vertex set and $E_n \subseteq \{(i,j) : i, j \in V_n\}$ denotes a random set of edges. Note that $G_n$ has $n+1$ vertices. For $v\in V_n$ let $D_{v}(G_n)$ denote the degree of vertex $v$ in the graph $G_n$.

There exist various versions of the preferential attachment model, each following slightly different conventions for adding new vertices. Here we consider the following model: the first graph $G_1$, also called the seed graph, consists of two vertices $v_0$ and $v_1$ connected by $m$ edges. For $t>1$, the graph $G_t$ is constructed by taking $G_{t-1}$ and adding one extra vertex $v_t$, that is connected to the vertices in $G_t$ by exactly $m$ edges. In the model we consider this process is better described by introducing a number of intermediate steps, described by graphs $G_{t,0}, G_{t,1}, \ldots G_{t,m}$, with vertex set $V_t=\{v_0,\ldots,v_t\}$. Begin by defining $G_{t,0}$ to be identical to $G_{t-1}$ together with an isolated vertex $v_t$. The graph $G_{t,1}$ is obtained by adding an edge between $v_t$ and one of the vertices in $V_{t-1}$ with probability proportional to $D_v(G_{t,0})+\delta(t)$. In general, for $i\in[m]$, we proceed by sampling vertex $v_{t,i}\in\{v_0,\ldots,v_{t-1}\}$ with conditional probability
\begin{equation}\label{eq:attachment_function_general}
\P\left(v_{t,i} = v\mid G_{t,i-1}\right)=\frac{D_{v}(G_{t,i-1})+\delta(t)}{\sum_{j=0}^{t-1} D_{j}(G_{t,i-1})+\delta(t)}\ ,
\end{equation}
and constructing $G_{t,i}$ by adding the edge $\{v_{t,i},v_t\}$ to $G_{t,i-1}$. Finally, define $G_t=G_{t,m}$. Note that the degree of $v_t$ in $G_t$ is exactly $m$.

The above model is rather general, as it allows for quite a bit of flexibility in terms of the function $\delta(t)$, as the only requirement is that $\delta(t)>-m$ to ensure that \eqref{eq:attachment_function_general} is indeed a valid probability. A classical choice is to take $\delta(t)$ as a constant, and the properties of the ensuing graphs are well studied (e.g., see \cite{VanderHofstad2017,VanderHofstad2020} and the references therein). However, we are interested in knowing when it is possible to distinguish graphs generated by a model where $\delta(t)$ is constant versus graphs generated by a model where $\delta(t)$ is a step function. The latter models a preferential attachment evolution, where at some point the characteristics of the attachment process change.

Consider the above model, and let $G_n$ denote the ``last'' graph obtained. This is our only observation, i.e., we do not have access to the sequence $\{G_1,\ldots,G_{n-1}\}$. In particular, the order of the vertices is unknown to us.
Clearly the distribution of this random graph is parameterized by $m$ and $\delta$. Since $G_n$ has exactly $n+1$ vertices and $nm$ edges, we have knowledge of $m$ (so this is not an unknown parameter). Therefore the only unknown parameter is the function $\delta$.

Our goal is to determine when one can find evidence in the final graph that the growth dynamics of the network has changed at some point. This can be rather naturally formulated as an hypothesis testing problem. Namely, we would like to conduct the following binary hypothesis test: under the null hypothesis (denoted by $H_0$) we assume $\delta(t)=\delta_0>-m$ for all $t\in\N$. Under the alternative hypothesis (denoted by $H_1$) we assume $\delta$ is a step function, namely
    \begin{equation}
    \delta(t) = \mathds 1\{t \leq \tau_n\}\delta_0 + \mathds 1\{t > \tau_n\}\delta_1\ ,
    \end{equation}
for some $\delta_0\neq\delta_1$ with $\delta_0,\delta_1>-m$, and $\tau_n\in\N$ with $\tau_n \leq n$.

Our main research goal is to determine when it is possible to distinguish the two hypotheses, based solely on $G_n$ (where we do not know the order in which the various vertices have arrived). For our first result we consider $\delta_0$ to be known, but we then relax this assumption and devise a test that does not require this knowledge while retaining the same asymptotic power characterization. In both cases we use the parameterization $\tau_n=n-cn^\gamma$, where $c>0$ and $\gamma\in(0,1)$, and obviously $c$ and $\gamma$ are unknown.

Note that the alternative model does coincide with the null model when either $\delta_1=\delta_0$ or $\tau_n=n$. Furthermore, since $\delta$ is a step-function, the attachment rule in \eqref{eq:attachment_function_general} can be further simplified to get, for $v\in\{v_0,\ldots,v_{t-1}\}$,

\begin{equation}
\label{eq:attachment_function_alt}
\P\left(v_{t,i} = v\mid G_{t,i-1}\right)=\left\{\begin{array}{ll}
\frac{D_{v}(G_{t,i-1})+\delta_0}{2(t-1)m + t\delta_0 + (i-1)} & \quad\text{if } t\leq \tau_n,\\\\
\frac{D_{v}(G_{t,i-1})+\delta_1}{2(t-1)m + t\delta_1 + (i-1)} & \quad\text{if } t>\tau_n.
\end{array}\right. \ 
\end{equation}

\subsection{Assumptions and notation} Throughout this paper, when limits are unspecified, they are taken as the graph size $n \to \infty$. We recall that we consider the parameterization $\tau_n=n-cn^\gamma$. All the parameters $m$, $\delta_0$, $\delta_1$, $c$, and $\gamma$ are assumed to remain constant as a function of $n$. We use the subscripts $0$ and $1$ in the expectation and probability operators to indicate whether we are considering the null or alternative hypothesis. Finally, we also use standard asymptotic notation: $a_n = \bigO(b_n)$ when $| a_n / b_n |$ is bounded, $a_n = \bigOmega(b_n)$ when $b_n = \bigO(a_n)$, $a_n = \smallO(b_n)$ when $a_n / b_n \to 0$, $a_n=\omega(b_n)$ when $b_n=o(a_n)$, and $a_n \asymp b_n$ when $a_n = (1 + \smallO(1)) b_n$. Also, we use the probabilistic versions of these: $a_n = \bigOp(b_n)$ when $| a_n / b_n |$ is stochastically bounded, $a_n = \bigOmegap(b_n)$ when $b_n = \bigOp(a_n)$, and $a_n = \smallOp(b_n)$ when $a_n / b_n$ converges to $0$ in probability.

\section{Minimal degree tests}\label{sec:minimal_degree_test}

All the tests we consider use the information contained in low-degree vertices, in particular the number of vertices of minimal degree. Although this is a rather simple idea, the number of minimal degree vertices is substantially affected  by the presence of a changepoint, even at late stages in the growth of the graph.

\subsection{Powerful test for known \texorpdfstring{$\delta_0$}{δ₀}}
We begin with the assumption that $\delta_0$ is known. Although this might seem unrealistic, it provides a wealth of information about the properties of the test statistic we consider, and paves the way for more general results when $\delta_0$ is unknown.

To define our test we first introduce some notation. To reduce the notational burden we identify the set of vertices $\{v_0,\ldots,v_n\}$ with $[n]\coloneqq \{0,1,\ldots,n\}$. Furthermore, let $D_v(t)\coloneqq D_v(G_t)$ denote the degree of vertex $v$ in graph $G_t$, and let $N_k(t)$ be the number of vertices of degree $k$ in the graph $G_t$, that is,
\begin{equation}
N_k(t) \coloneqq \sum_{v\in[t]} \mathds 1\{D_v(t) = k\}\ .
\end{equation}
Since each vertex is attached to at least $m$ other vertices, in our model we naturally have that $N_k(n) = 0$ for $k < m$, and $N_m(n)$ denotes the number of vertices with minimal degree. The latter quantity plays a crucial role in our test.

It is well-known that in the classical preferential attachment model, which corresponds to our null model, the number of vertices of degree $k\geq m$ is highly concentrated \cite{Deijfen2007,VanderHofstad2017}. In particular, $N_k(n)$ is well concentrated around $n p_k(\delta_0)$ where $p_k=p_k(\delta_0)$ satisfies the recursion
\begin{equation}\label{eqn:recursion_p_k}
p_k=\frac{k-1+\delta_0}{2+\delta_0/m}p_{k-1}-\frac{k+\delta_0}{2+\delta_0/m}p_k\ ,
\end{equation}
for $k>m$ with
\begin{equation}
\label{eq:limiting_degree_distribution}
p_m(\delta_0) = \frac{2 + \delta_0 / m}{m + \delta_0 + 2 + \delta_0 / m}\ .
\end{equation}
This recursion is easily solved, giving rise to the following expression for $p_k(\delta_0)$:
\begin{equation}\label{eqn:p_k}
p_k(\delta_0)\coloneqq(2+\delta_0/m)\frac{\Gamma(k+\delta_0)\Gamma(m+2+\delta_0+\delta_0/m)}{\Gamma(m+\delta_0)\Gamma(k+3+\delta_0+\delta_0/m)}\ .
\end{equation}
Thus, $p_k(\delta_0)$ is the limiting degree distribution of the random graph $G_n$ under the null model.

We are now able to introduce our test statistic, that simply compares the number of minimal degree vertices to its asymptotic expected value under the null model, as
\begin{equation}
T(G_n)\coloneqq N_m(n) - n p_m(\delta_0)\ .
\end{equation}
If the observed value of $T(G_n)$ is significantly different from zero, then we have evidence to reject the null model. This brings us to the first result, characterizing when such a test is powerful. Specifically we introduce a test based on this statistic that is guaranteed to have type-I error that is at most $\alpha$ (asymptotically), and that is asymptotically powerful provided $\gamma>\tfrac{1}{2}$. In other words, under the alternative hypothesis the type-II error converges to zero provided $\gamma>\tfrac{1}{2}$. When $\gamma<\tfrac{1}{2}$, the test is powerless, and when $\gamma=\tfrac{1}{2}$, the type-II error is bounded by a constant that depends on the specific model parameters.

Although this result indicates when the test is powerful or powerless, it provides only a conservative upper bound on the type I error. It follows from \cite{Baldassarri2021} that this test statistic is asymptotically normal, and therefore we can calibrate this test to guarantee that the type I error is $(1+\smallO(1))\alpha$ as $n\to\infty$. We do this in Section \ref{sec:asymptotically_calibrated_tests}. The proof of Theorem~\ref{thm:minimal_degree_test_known_delta} is given in Section~\ref{sec:proof_known}.

\begin{theorem}[Asymptotically powerful test: known $\delta_0$]
\label{thm:minimal_degree_test_known_delta}
Consider the test that rejects the null hypothesis for large values of $T(G_n)$. Namely, define the test
\begin{equation}\label{eq:minimal_degree_test_known_delta}
\psi(G_n)\coloneqq \mathds 1{\left\{|T(G_n)|\geq m \sqrt{8n \log(2 / \alpha)}\right\}}\ ,
\end{equation}
where $\alpha\in(0,1)$. The type-I error of this test is asymptotically bounded by $\alpha$, i.e.,
\begin{equation}
\P_0\left(\psi(G_n)\neq 0\right)\leq (1 + \smallO(1)) \alpha\ .
\end{equation}
Furthermore, the type-II error of this test satisfies
\begin{align}
\lefteqn{\P_1\left(\psi(G_n)= 0\right)}\\
&\leq \left\{\begin{array}{ll}
\smallO(1) & \text{ when } \gamma > \tfrac{1}{2} ,\\
(2 + \smallO(1)) \exp\left(-\left(\left(\frac{c  |1 - p_m(\delta_0) / p_m(\delta_1)|}{m \sqrt{8}} - \sqrt{\log(2 / \alpha)}\right) \vee 0\right)^2\right) & \text{ when } \gamma = \tfrac{1}{2},
\end{array}\right.\ 
\end{align}
and
$\P_1\left(\psi(G_n)= 0\right)\geq(1+\smallO(1))(1-\alpha)$
when $\gamma<\tfrac{1}{2}$.
\end{theorem}

The proof of this result is based on the following observations. Under the null model it is known that $\E_0(N_m(n))-np_m(\delta_0)=\bigO(1)$. Under the alternative model we can show that
\begin{equation}
\E_1\left[N_m(n)\right] - np_m(\delta_0)= (1+\smallO(1)) \eta(\delta_0,\delta_1) n^\gamma\ ,
\end{equation}
where
\begin{equation}
\eta(\delta_0,\delta_1)\coloneqq c (1 - p_m(\delta_0) / p_m(\delta_1))\ .\label{eqn:shift_T}
\end{equation}
Therefore there is a substantial difference in the expected values of the test statistic under the null and alternative models. Note that both $\E_1\left[N_m(n)\right]$ and $np_m(\delta_0)$ have the same order of magnitude $\bigO(n)$, so the above result characterizes the second-order behavior of $\E_1\left[N_m(n)\right]$ and thus is somewhat delicate. Besides characterizing the mean of $N_m(n)$, we must also characterize the fluctuations of $N_m(n)$ around it. These are of small order, and controlled by a rather standard application of Azuma-Hoeffding's inequality. Specifically $N_m-\E[N_m(n)]=\bigOp(\sqrt{n})$. This result holds both under the null and alternative hypothesis, showing that it is possible to construct a powerful test when $\gamma>\tfrac{1}{2}$.

As the proposed test is powerless when $\gamma<\tfrac{1}{2}$, one might wonder if there is any test that can have power in that situation. Although a formal answer to this question is still open, we conjecture that no test can be powerful in that scenario:
\begin{conjecture}[Powerless tests when $\gamma<\tfrac{1}{2}$]\label{conj:lower_bounds}
Consider the case $\gamma<\tfrac{1}{2}$. We conjecture the following:
\begin{enumerate}
\item[(i)] All tests based on the vertex degrees $\{D_v(n)\}_{v\in[n]}$ are powerless.
\item[(ii)] All tests based on $G_n$ are powerless.
\end{enumerate}
\end{conjecture}
Obviously the statement (ii) implies (i). The main motivation for (i) is that when $\gamma<\tfrac{1}{2}$ the number of vertices with degree $k$ will deviate from $np_k(\delta_0)$ by at most $\bigO(n^\gamma)$. These deviations become smaller when $k$ gets larger. On the other hand, the fluctuations of $N_k(n)$ around its mean will also become smaller, but should always be of higher order than this mean-shift. Actually, the results in \cite{Baldassarri2021} characterize the joint distribution of the degree counts under the null model and show this is asymptotically a multi-variate normal distribution. As shown in Lemma~\ref{lem:joint_CLT} these degree counts are also asymptotically normal, with the same covariance matrix, but a rather small mean-shift. We conjecture this shift is small enough so to imply the total variation distance between the two distributions is close to zero, implying (i). The conjecture (ii) is significantly stronger, stating that higher-level information contained on the edge structure of $G_n$ will not be helpful for this testing problem. This is expected given that the attachment dynamics are only driven by the vertex degrees, but proving such a statement requires a careful formalization of this insight.

\subsection{Powerful test for unknown \texorpdfstring{$\delta_0$}{δ₀}}

The knowledge of $\delta_0$ was crucial for the test above, as it gives a benchmark to compare $N_m(n)$ against, namely $np_m(\delta_0)$. Without this knowledge we must essentially estimate, from $G_n$, the value of $\delta_0$. For such an approach to be fruitful a candidate estimator must be ``close enough'' to $\delta_0$ both under the null and alternative models. Gao and van der Vaart in \cite{Gao2017} consider the problem of estimating $\delta_0$ when the preferential attachment function is \emph{constant}, meaning that we are under the null model in our formulation. The authors proposed a maximum likelihood estimator based on $G_n$, and showed it consistently estimates $\delta_0$ and is asymptotically normal. A natural idea is to start by considering this estimator as well, and understand how its properties change when $G_n$ is generated under the alternative model.

As done in \cite{Gao2017}, to avoid the usual issues at the boundary of the parameter space we make an extra assumption that the range of possible values for $\delta_0$ and $\delta_1$ is known:
\begin{assumption}[Containment of $\delta_0,\delta_1$]\label{ass:parameter_space}
Let $-m<\delta_{\min}\leq \delta_{\max}$ be known, and assume that $\delta_0,\delta_1\in(\delta_{\min},\delta_{\max})$.
\end{assumption}
As shown in \cite{Gao2017}, under the null model the (normalized) log-likelihood function $\iota_n:[\delta_{\min},\delta_{\max}]\to\R$ is given by
\begin{align}
\iota_n(\delta)
\coloneqq{}& \frac{1}{n+1}\left(\sum_{k=1}^\infty \log(k+\delta) \left(N_{>k}(n)-(n+1)\mathds 1\{k<m\}\right)\ -\ \sum_{t=2}^n\sum_{i=1}^m \log S_{t,i-1}(\delta)\right)\notag\\
={}& \frac{1}{n+1}\sum_{k=m}^\infty \log(k+\delta) N_{>k}(n)\ -\ \frac{1}{n+1}\sum_{t=2}^n\sum_{i=1}^m \log S_{t,i-1}(\delta)\ ,
\end{align}
where $S_{t,i-1}(\delta)\coloneqq t\delta+2m(t-1)+(i-1)$ and $N_{>k}(n)\coloneqq \sum_{j>k} N_j(n)$. The maximum-likelihood estimator is defined as
\begin{equation}\label{eqn:MLE}
\hat\delta_n\coloneqq\argmax_{\delta\in[\delta_{\min},\delta_{\max}]} \iota_n(\delta)\ .
\end{equation}
Equivalently (for large $n$) we can define $\hat\delta_n$ as the solution in $\delta\in[\delta_{\min},\delta_{\max}]$ of $\frac{\partial}{\partial \delta} \iota_n(\delta)\coloneqq \iota'_n(\delta)=0$. Although not obvious, this definition coincides with \eqref{eqn:MLE} for large $n$, since it is shown in \cite{Gao2017} that the solution of $\iota'_n(\delta)=0$ exists and is unique for large enough $n$ with high probability. Note that the score function is given by
\begin{equation}\label{eqn:iota_n_prime}
\iota'_n(\delta)=\frac{1}{n+1}\sum_{k=m}^\infty \frac{1}{k+\delta}N_{>k}(n)\ -\ \frac{1}{n+1}\sum_{t=2}^n\sum_{i=1}^m \frac{t}{S_{t,i-1}(\delta)}\ .
\end{equation}

Motivated by this estimator we consider the test statistic
\begin{equation}\label{eqn:test_statistic_unknown}
Q(G_n)\coloneqq N_m(n)-np_m(\hat\delta_n)\ .
\end{equation}
This is analogous to the previously considered statistic, with the exception that $\delta_0$ is replaced by the above estimator.

A test based on the above statistic will only be sensible if $\hat\delta_n$ is a good surrogate for $\delta_0$, under both the null and alternative models. When $\tfrac{1}{2}<\gamma<1$ this is indeed the case, and we show that $\hat\delta_n$ is a consistent estimator of $\delta_0$ under both the null and alternative models. However, consistency is not enough, and it is necessary to carefully characterize the rate of convergence of this estimator under the alternative model. It turns out that the deviations of $N_m(n)$ and $np_m(\hat\delta_n)$ around their respective means have exactly the same order, but with \textit{different} leading constants. A careful characterization of those constants is the crucial result leading to our main result:
\begin{theorem}[Asymptotically powerful test, unknown $\delta_0$]\label{thm:main2}
Consider Assumption~\ref{ass:parameter_space} and the test statistic defined in \eqref{eqn:test_statistic_unknown}.
Let $a_n$ be a positive diverging sequence such that $a_n=\omega(\sqrt{n}\log n)$ and $a_n=n^{\tfrac{1}{2}+o(1)}$ and define the test
\begin{equation}
\phi(G_n)\coloneqq \mathds 1{\left\{|Q(G_n)|\geq a_n\right\}}\ .
\end{equation}
The type-I error of this test converges to zero as $n\to\infty$, i.e., 
\begin{equation}
\P_0\left(\phi(G_n)\neq 0\right)=\smallO(1)\ .
\end{equation}
Furthermore, when $\delta_0\neq\delta_1$ and $\tfrac{1}{2}<\gamma<1$ this test has vanishing type II error, i.e., as $n\to\infty$,
\begin{equation}
\P_1\left(\phi(G_n)=0\right)=\smallO(1)\ .
\end{equation}
\end{theorem}

The proof of the theorem is deferred to Section~\ref{sec:proof_unknown} and it is rather involved. It builds upon some of the results used to prove Theorem~\ref{thm:minimal_degree_test_known_delta}, namely the characterization of $N_m(n) - np_m(\delta_0)$. However, it does require a very careful characterization of $np_m(\hat\delta_n)-np_m(\delta_0)$. It turns out that both quantities have essentially the same order of magnitude, namely $\bigOp(n^{\gamma})$. However, the leading constants are different, and this fact allows the test to be powerful in the regime $\tfrac{1}{2}<\gamma<1$. 

The cases $\gamma=\tfrac{1}{2}$ and $\gamma=1$ are special. For $\gamma=\tfrac{1}{2}$, one can expect Gaussian fluctuations of $N_m(n)$ under the null and alternative hypotheses to compete with the resulting change in expectations of $N_m(n)$, so the Type-II error can not be expected to vanish, as it does for $\gamma\in (\tfrac{1}{2},1)$. For $\gamma=1$, on the other hand, both the maximal and minimal degree tests seem to perform well, but it is unclear which performs best, or whether there even is a better test available.

\subsection{Asymptotically calibrated tests}\label{sec:asymptotically_calibrated_tests}

The two theorems above characterize the regime when the proposed tests are powerful. However, they fall short of providing guidelines to properly calibrate the tests. Particularly, in a fixed significance testing framework, for any $\alpha\in(0,1)$ we would like to ensure that under the null model the type-I error is approximately $\alpha$. Theorem~\ref{thm:minimal_degree_test_known_delta} provides only an rather conservative asymptotic upper bound on the type-I error, due to the worst-case nature of the Azuma-Hoeffding inequality. To introduce our calibrated test, we define
\begin{align}%
\label{w-var-def}
w(\delta_0,m)&\coloneqq \frac{m^2 (m+\delta_0) (1 + m + \delta_0) (2m+\delta_0)}{(\delta_0 + 2m(1+m+\delta_0)) (\delta_0 + m(2+m+\delta_0))^2}\ .
\end{align}%
Furthermore, let $z_\alpha$ denote the right-quantile function of the standard normal distribution\footnote{For $\alpha\in(0,1)$, let $z_\alpha$ be the unique solution of $\alpha=\int_{z_\alpha}^\infty \frac{1}{\sqrt{2\pi}}{\rm e}^{-\frac{z^2}{2}}dz$.}. In particular, $z_{\alpha/2}>0$ when $\alpha\in(0,1)$. We are now ready to present the asymptotically calibrated test when $\delta_0$ is assumed known.
\begin{theorem}[Asymptotically calibrated test for known $\delta_0$]
\label{thm:asymptotically_calibrated_minimal_degree_test_known_delta}
Let $\alpha\in(0,1)$ and define the test
\begin{equation}\label{eq:asymptotically_calibrated_minimal_degree_test_known_delta}
\psi_{\rm cal}(G_n)\coloneqq \mathds 1{\left\{|T(G_n)|\geq \sqrt{n w(\delta_0,m)}z_{\alpha/2}\right\}}\ .
\end{equation}
As $n\to\infty$, the type-I error of this test converges to $\alpha$:
\begin{equation}
\P_0\left(\psi_{\rm cal}(G_n)\neq 0\right)\to \alpha\ .
\end{equation}
Furthermore, when $\delta_0\neq\delta_1$ and $\tfrac{1}{2}<\gamma<1$ this test has vanishing type-II error as $n\rightarrow \infty$:
\begin{align}
\P_1\left(\psi_{\rm cal}(G_n)= 0\right)\to 0\ .
\end{align}
\end{theorem}%
A statement regarding partial power when $\gamma=\frac{1}{2}$ is also possible, but not particularly insightful. To define the next test we require some additional notation. Let
\begin{align}%
\label{v-var-def}
v(\delta_0,m)&\coloneqq \sum_{k=m}^\infty \frac{m p_k(\delta_0)}{(k+\delta_0)(2m+\delta_0)} - \frac{m}{(2m+\delta_0)^2}\ ,
\end{align}%
and
\begin{align}
\label{u-var-def}
u(\delta_0,m)&\coloneqq - \frac{m^4}{v(\delta_0,m) (\delta_0 + m(2 + m + \delta_0))^4}\ .
\end{align}
The next theorem presents the asymptotically calibrated test for the unknown $\delta_0$ case:

\begin{theorem}[Asymptotically calibrated test, unknown $\delta_0$]\label{thm:asymptotically_calibrated_minimal_degree_test_unknown_delta}
Consider Assumption~\ref{ass:parameter_space}, let $\alpha\in(0,1)$, and define the test
\begin{equation}
\phi_{\rm cal}(G_n)\coloneqq \mathds 1{\left\{|Q(G_n)|\geq \sqrt{n (w(\hat{\delta}_n,m)+u(\hat{\delta}_n,m)) } z_{\alpha/2} \right\}}\ .
\end{equation}
As $n\to\infty$, the type-I error of this test converges to $\alpha$:
\begin{equation}
\P_0\left(\phi_{\rm cal}(G_n)\neq 0\right)\to\alpha\ .
\end{equation}
Furthermore, when $\delta_0\neq\delta_1$ and $\tfrac{1}{2}<\gamma<1$ this test has vanishing type-II error as $n\rightarrow \infty$:
\begin{equation}
\P_1\left(\phi_{\rm cal}(G_n)=0\right)\to0\ .
\end{equation}
\end{theorem}
The proofs of Theorems \ref{thm:asymptotically_calibrated_minimal_degree_test_known_delta} and \ref{thm:asymptotically_calibrated_minimal_degree_test_unknown_delta} are an immediate consequence of the asymptotic normality of the test statistics as discussed in the next section, together with the consistency of $\hat\delta_n$ as an estimator of $\delta_0$.

\section{Asymptotic normality of test statistics}

In this section we characterize the asymptotic distribution of the proposed test statistics which allowed us to calibrate the tests in Section \ref{sec:asymptotically_calibrated_tests}.

When $\delta_0$ is known, the situation is relatively simple. Under the null model and when $m=1$, it is known that $N_m(n)$ with $k \geq m$ admits a central limit theorem \cite{Samorodnitsky2016}. In particular, this shows that $N_m(n)$ is asymptotically normally distributed. Furthermore, \cite{Baldassarri2021} extends these results to the general case $m\geq 1$, making them applicable to our setting. For the case of unknown $\delta_0$ the situation is a bit more complicated.  Recall that our test statistic is $Q(G_n) = N_m(n)-n p_m(\hat\delta_n)$. It is known from \cite{Gao2017} that, under the null model, $\hat\delta_n$ is asymptotically normal. This does not, however, immediately imply that $Q(G_n)$ is also asymptotically normal under the null model. The results in \cite{Baldassarri2021} establish that $(N_m(n),N_{m+1}(n),\ldots)$ is asymptotically normal (under the null model), strongly hinting at asymptotic normality of $Q(G_n)$. Furthermore, even under the alternative model, one might expect asymptotic normality of the test statistics, with exactly the same asymptotic variance, as the number of vertices that enter after the (late) changepoint is too small to change the asymptotic variance.

Let $\mathcal{N}(\mu,\sigma^2)$ denote the normal distribution with mean $\mu$ and variance $\sigma^2$, and let $\xrightarrow{\smash{\raisebox{-1.5pt}{$\scriptstyle{}D$}}}$ denote convergence in distribution. The following theorem, proved in Section~\ref{sec:proof_asymptotic_normality_test_statistics}, establishes the asymptotic properties of the test statistics.

\begin{theorem}[Asymptotic normality of test statistics]\label{thm:asymptotic_normality_test_statistics}
Recall the definitions of $w$ and $u$ in \eqref{w-var-def} and \eqref{u-var-def} respectively. As $n\to\infty$,
\begin{align}\label{eq:asymptotic_normality_known_delta_test_statistic_null_hypothesis}%
\frac{T(G_n)}{\sqrt{n}} \xrightarrow{D} \mathcal{N}\left(0,w(\delta_0,m)\right)\ .
\end{align}%
Moreover, under Assumption~\ref{ass:parameter_space} and the null model, as $n\to\infty$,
\begin{align}%
\frac{Q(G_n)}{\sqrt{n}}=\frac{N_m(n)-np_m(\hat\delta_n)}{\sqrt{n}}\ \xrightarrow{D} \mathcal{N}\left(0,w(\delta_0,m)+u(\delta_0,m)\right)\ .\label{eq:asymptotic_normality_unknown_delta_test_statistic}
\end{align}%
Furthermore, under the alternative model with $\gamma\in(0,1)$, as $n\to\infty$,
\begin{align}\label{eq:asymptotic_normality_known_delta_test_statistic_alternative_hypothesis}%
\frac{T(G_n)-\E_1[T(G_n)]}{\sqrt{n}}\xrightarrow{D} \mathcal{N}\left(0,w(\delta_0,m)\right)\ .
\end{align}%
Moreover, under the alternative model with $\gamma\in(\tfrac{1}{2},1)$ and Assumption~\ref{ass:parameter_space}, as $n\to\infty$,
\begin{align}\label{eq:asymptotic_normality_unknown_delta_test_statistic_alternative_hypothesis}%
\frac{Q(G_n)-\E_1[Q(G_n)]}{\sqrt{n}}\xrightarrow{D} \mathcal{N}\left(0,w(\delta_0,m)+u(\delta_0,m)\right)\ .
\end{align}%
\end{theorem}
Note that \eqref{eq:asymptotic_normality_known_delta_test_statistic_null_hypothesis} has been proved in \cite{Baldassarri2021}. In the second statement \eqref{eq:asymptotic_normality_unknown_delta_test_statistic}, $u(\delta_0,m)$ captures the adjustment in the variability in the test statistic by using an estimate of $\delta_0$, instead of the actual value. The convergence results \eqref{eq:asymptotic_normality_known_delta_test_statistic_null_hypothesis} and \eqref{eq:asymptotic_normality_unknown_delta_test_statistic} provide an avenue for asymptotically exact calibration of the proposed test as given in Theorems~\ref{thm:asymptotically_calibrated_minimal_degree_test_known_delta} and \ref{thm:asymptotically_calibrated_minimal_degree_test_unknown_delta} (equivalently, for computation of asymptotically exact $p$-values).

To establish asymptotic normality under the alternative model, we exploit the corresponding statements under the null model, together with a correction term quantifying the effects due to the presence of a changepoint (partially relying on the arguments in Theorems~\ref{thm:minimal_degree_test_known_delta} and \ref{thm:main2}). It is important to remark that, for the convergence result \eqref{eq:asymptotic_normality_unknown_delta_test_statistic_alternative_hypothesis}, the assumption $\gamma>\tfrac{1}{2}$ appears to be merely technical, and it should be possible to drop it. Finally, note that the asymptotic variances remain the same,
whether one is considering the null or alternative model, as expected. The proof of Theorem \ref{thm:asymptotic_normality_test_statistics} is carried out in Section~\ref{sec:proof_asymptotic_normality_test_statistics}.

\section{Numerical experiments}

In this section we examine the properties of the proposed tests using simulation. This serves a two-fold purpose, namely to empirically assess the validity of the theoretical guarantees given, as well as to investigate the finite-sample properties of the tests.

To assess the finite sample properties of the asymptotically calibrated tests we conduct several numerical experiments. For simplicity we fix $\delta_0=0$ for all the experiments (this is the so-called linear preferential attachment model in \cite{Barabasi1999}). We take $m=5$, $c=1$, $\gamma=\frac{3}{4}$, and $\delta_1\in\{-1,0,1\}$. Note that $\delta_1=\delta_0=0$ corresponds to the null model. We consider graphs of different sizes, namely $n\in\{1000,2000,5000,10000,20000,50000,100000,200000\}$, and for each value of $n$ we generate $B=2000$ independent graphs using the preferential attachment model in \eqref{eq:attachment_function_general}. Specifically, for the three cases $\delta_1=\{-1,0,1\}$ and each value $n$ we obtain $\{g_n^{(b)}\}_{b=1}^B$ graphs.

Note that as $n$ increases the relative distance between the changepoint and $n$ decreases considerably, and only a minute part of the graph is affected after the changepoint. For instance, for $n=5000$ there are only 594 vertices that join the graph after the changepoint, so about 12\% of the vertices. However, for $n=200000$ only 9457 vertices join after the changepoint, a mere 4.7\% of the total vertices.

\begin{figure}
\centering
\subfigure[Power of the tests (known $\delta_0$ and $\alpha=0.05$).]{\includegraphics[width=0.45\textwidth]{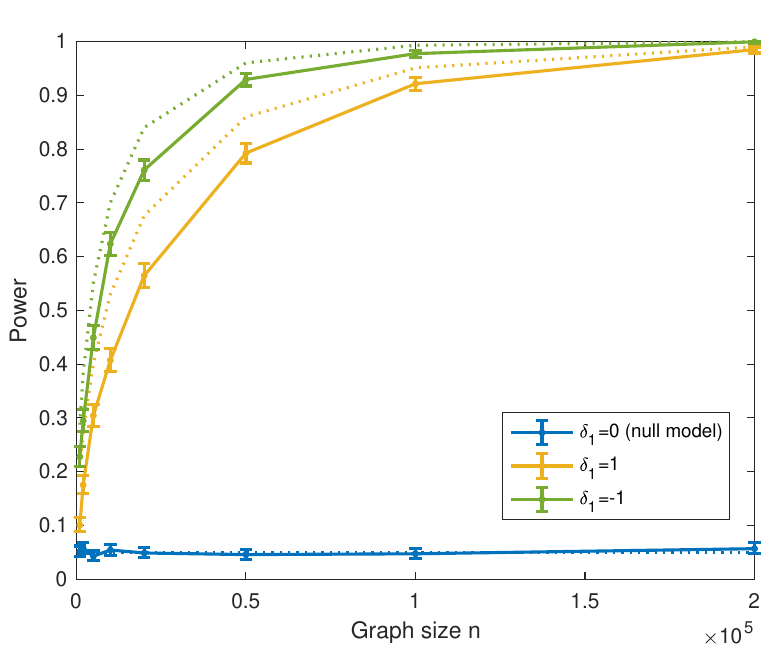}}
\hspace{0.05\textwidth}
\subfigure[Power of the tests (unknown $\delta_0$ and $\alpha=0.05$).]{\includegraphics[width=0.45\textwidth]{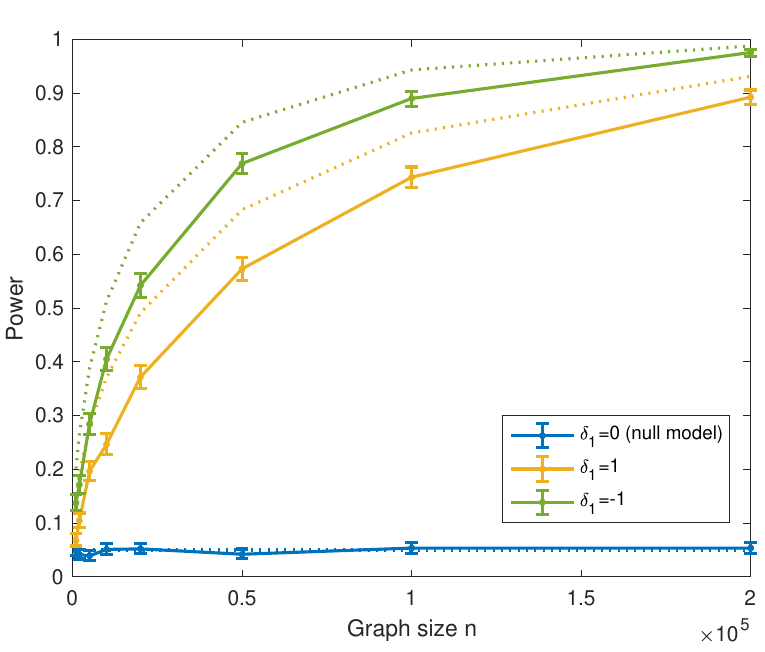}}
\caption{Power of the tests calibrated using asymptotic normality for $m=5$. Panel (a) corresponds to the case of known $\delta_0$ and panel (b) corresponds to the unknown $\delta_0$ case. Both plots consider a fixed significance level testing with level $\alpha=0.05$. The solid lines correspond to the power estimated using $B=2000$ graph samples, and (the pointwise 95\% confidence bands are computed using the Clopper-Pearson exact method. The dashed lines are estimates of the power based on the asymptotic characterization of the test statistics.}\label{fig:power}
\end{figure}

In Figure~\ref{fig:power} we depict the power of the two proposed tests. For concreteness we consider fixed significance testing at level $\alpha=0.05$ (qualitatively the results are similar for other significance levels). We clearly see that both tests are well calibrated, even for small values of $n$. Also, as expected, the power increases as a function of $n$. As intuitively expected, the test that does not assume knowledge of $\delta_0$ has slightly less power than the test making use of that knowledge. Finally, there is an asymmetry of the power depending on whether $\delta_1>\delta_0$ or $\delta_1<\delta_0$, the later scenario leading to higher power. This is expected, as for smaller $\delta_1$ values the empirical degree distribution has heavier tails, and detecting the presence of a changepoint becomes easier.

\begin{figure}
\centering
\subfigure[$T(G_n)$ with $n=100$ and $\delta_0 = \delta_1 = 0$.]{\includegraphics[width=0.45\textwidth]{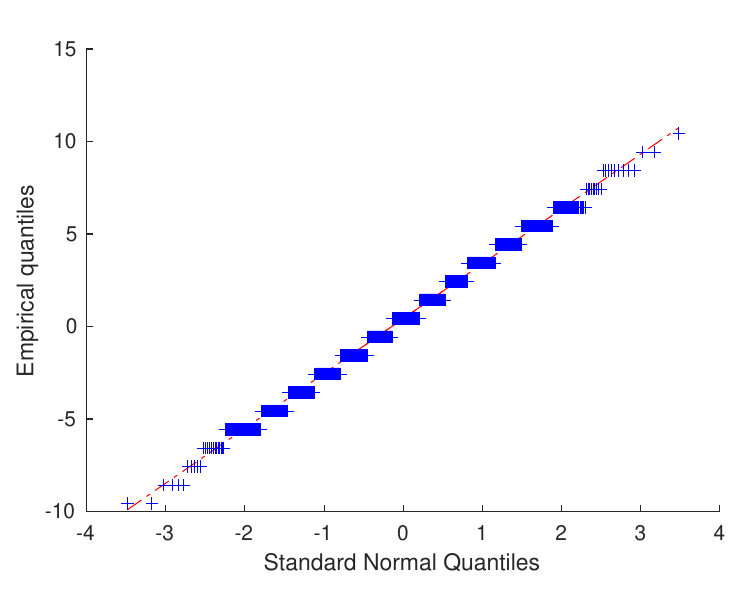}}
\hspace{0.05\textwidth}
\subfigure[$T(G_n)$ with $n=1000$ and $\delta_0 = \delta_1 = 0$.]{\includegraphics[width=0.45\textwidth]{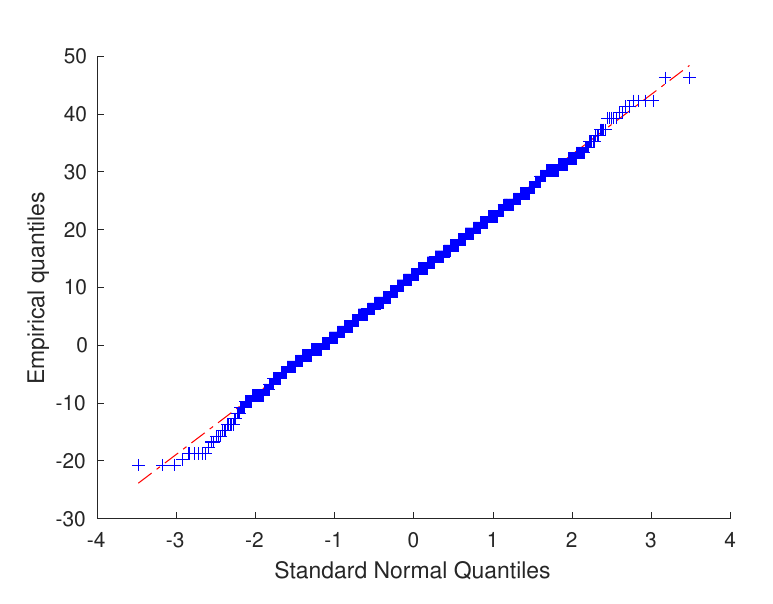}}
\subfigure[$Q(G_n)$ with $n=100$ and $\delta_0 = \delta_1 = 0$.]{\includegraphics[width=0.45\textwidth]{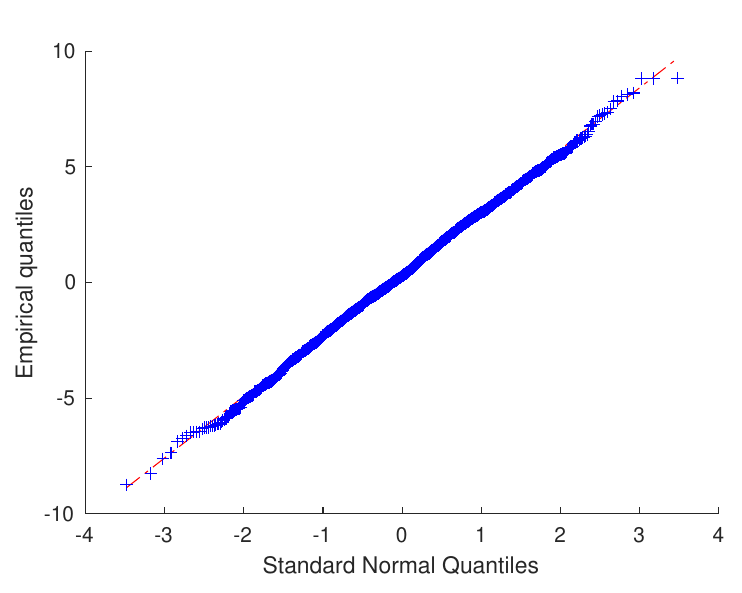}}
\hspace{0.05\textwidth}
\subfigure[$Q(G_n)$ with $n=1000$ and $\delta_0 = \delta_1 = 0$.]{\includegraphics[width=0.45\textwidth]{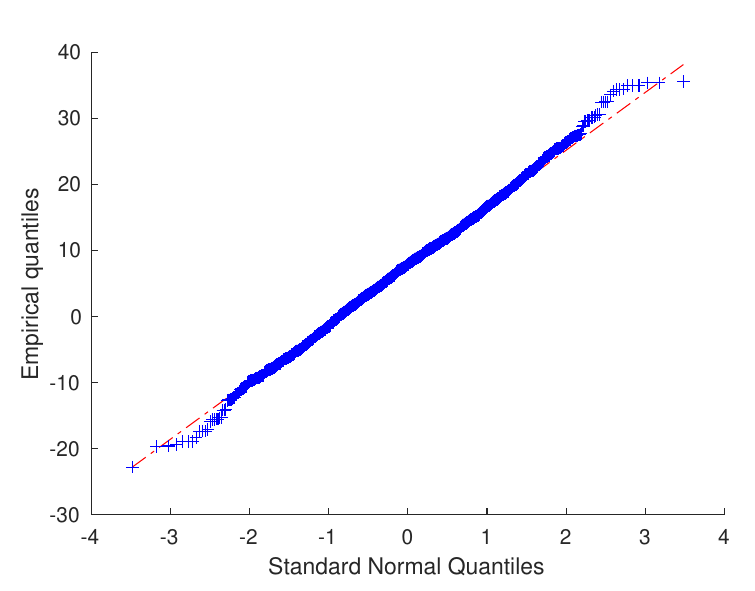}}
\caption{Normal QQ plots of the test statistics under the null model, for $n\in\{100,1000\}$. Panels (a) and (b) correspond to $T(G_n)$ (known $\delta_0$ statistic) and panels (c) and (d) correspond to $Q(G_n)$ (unknown $\delta_0$ statistics). Even for very small values of $n$ distribution of the statistics is very well approximated by a normal distribution.}\label{fig:asym_norm}
\end{figure}

To illustrate the results of Theorem~\ref{thm:asymptotic_normality_test_statistics}, in Figure~\ref{fig:asym_norm} we display normal quantile-quantile plots of the test statistics for very small values of $n$. We see that, even for $n=100$ (so, a graph with $nm=500$ edges) both test statistics are approximately normally distributed. In order to have a better assessment of the mean and variance of the test statistics we compute (for $\delta_1\in\{-1,0,1\}$) the first two empirical moments, and compare them with their asymptotic counterparts, both under the null and alternative hypothesis. Define
$$\hat\mu^{\sss (T)}_n \coloneqq \frac{1}{B}\sum_{b=1}^B T(g_n^{\sss (b)})$$
and 
$$\hat v^{\sss (T)}_n \coloneqq \frac{1}{B}\sum_{b=1}^B (T(g_n^{\sss (b)})-\hat\mu^{\sss (T)}_n)^2\ ,$$
and the analogous definitions of $\hat\mu^{\sss (Q)}_n$ and $\hat v^{\sss (Q)}_n$. In Table~\ref{tbl:asymptotic_mean_var} we compare these values (adequately rescaled) with the expected asymptotic values. As can be seen, the empirical variance closely matches the asymptotic variance, both under the null and alternative models, even for small values of $n$. This is in agreement with Theorem~\ref{thm:asymptotic_normality_test_statistics}. For the mean of the statistics, one sees that the scaling by $n^\gamma$ and the corresponding leading constant is also accurate, but finite sample effects are more evident under the alternative model when $n$ is small.

\begin{table}
\centering\small
\caption{Numerical estimates of (rescaled) mean and variance of the test statistics, and comparison with corresponding asymptotic values. The estimates above are based on $B=2000$ independently generated graphs.}\label{tbl:asymptotic_mean_var}
\subfigure[Null model $\delta_1=\delta_0=0.$]{
\begin{tabular}{p{1.2cm}<{\raggedleft}p{2.2cm}<{\raggedleft}p{2.2cm}<{\raggedleft}p{2.2cm}<{\raggedleft}p{3.2cm}<{\raggedleft}}
\toprule
$n$     & $\hat\mu_n^{\sss (T)}$ & $\hat\mu_n^{\sss (Q)}$ & $\hat v_n^{\sss (T)}/n$ & $\hat v_n^{\sss (Q)}/n$\\
\midrule
1000    &   $0.0127$  &   $0.0105$  &   $0.1014$  &   $0.0801$\\
2000    &   $0.0017$  &   $0.0042$  &   $0.1016$  &   $0.0795$\\
5000    &   $0.0056$  &   $0.0041$  &   $0.1009$  &   $0.0789$\\
10000   &   $-0.0052$ &   $-0.0038$ &   $0.1049$  &   $0.0831$\\
20000   &   $0.0007$  &   $0.0012$  &   $0.1034$  &   $0.0828$\\
50000   &   $0.0049$  &   $-0.0009$ &   $0.0986$  &   $0.0780$\\
100000  &   $-0.0066$ &   $-0.0054$ &   $0.1000$  &   $0.0814$\\
200000  &   $0.0031$  &   $0.0046$  &   $0.1048$  &   $0.0811$\\
\midrule
        &   $\eta(\delta_0,\delta_0)$   &   $\alpha(\delta_0,\delta_0)$ & $w(\delta_0,m)$ & $w(\delta_0,m)+u(\delta_0,m)$\\
        &   $=0$                        &   $=0$                        &   $=0.1020$     &   $=0.0811$\\
\bottomrule   
\end{tabular}
}
\subfigure[Alternative model $\delta_0=0$ and $\delta_1=1.$]{
\begin{tabular}{p{1.2cm}<{\raggedleft}p{2.2cm}<{\raggedleft}p{2.2cm}<{\raggedleft}p{2.2cm}<{\raggedleft}p{3.2cm}<{\raggedleft}}
\toprule
$n$ & $\hat\mu_n^{\sss (T)}/n^\gamma$ & $\hat\mu_n^{\sss (Q)}/n^\gamma$ & $\hat v_n^{\sss (T)}/n$ & $\hat v_n^{\sss (Q)}/n$\\
\midrule
1000    &   $-0.0460$ &   $-0.0276$ &   $0.0944$  &   $0.0772$\\
2000    &   $-0.0487$ &   $-0.0308$ &   $0.0992$  &   $0.0776$\\
5000    &   $-0.0554$ &   $-0.0371$ &   $0.1017$  &   $0.0826$\\
10000   &   $-0.0558$ &   $-0.0375$ &   $0.0994$  &   $0.0799$\\
20000   &   $-0.0574$ &   $-0.0395$ &   $0.1041$  &   $0.0799$\\
50000   &   $-0.0594$ &   $-0.0409$ &   $0.1033$  &   $0.0832$\\
100000  &   $-0.0610$ &   $-0.0421$ &   $0.1042$  &   $0.0809$\\
200000  &   $-0.0610$ &   $-0.0427$ &   $0.0974$  &   $0.0759$\\
\midrule
        &   $\eta(\delta_0,\delta_1)$   &   $\alpha(\delta_0,\delta_1)$ & $w(\delta_0,m)$ & $w(\delta_0,m)+u(\delta_0,m)$\\
        &   $=-0.0649$                  &   $=-0.0464$                  &   $=0.1020$     &   $=0.0811$\\
\bottomrule   
\end{tabular}
}
\subfigure[Alternative model $\delta_0=0$ and $\delta_1=-1.$]{
\begin{tabular}{p{1.2cm}<{\raggedleft}p{2.2cm}<{\raggedleft}p{2.2cm}<{\raggedleft}p{2.2cm}<{\raggedleft}p{3.2cm}<{\raggedleft}}
\toprule
$n$ & $\hat\mu_n^{\sss (T)}/n^\gamma$ & $\hat\mu_n^{\sss (Q)}/n^\gamma$ & $\hat v_n^{\sss (T)}/n$ & $\hat v_n^{\sss (Q)}/n$\\
\midrule
1000    &  $0.0672$  &   $0.0436$  &   $0.1087$  &   $0.0831$\\
2000    &  $0.0679$  &   $0.0443$  &   $0.1020$  &   $0.0760$\\
5000    &  $0.0700$  &   $0.0463$  &   $0.1033$  &   $0.0799$\\
10000   &  $0.0725$  &   $0.0491$  &   $0.0952$  &   $0.0759$\\
20000   &  $0.0713$  &   $0.0488$  &   $0.1049$  &   $0.0817$\\
50000   &  $0.0738$  &   $0.0509$  &   $0.1043$  &   $0.0833$\\
100000  &  $0.0743$  &   $0.0515$  &   $0.1057$  &   $0.0875$\\
200000  &  $0.0750$  &   $0.0522$  &   $0.0998$  &   $0.0793$\\
\midrule
        &   $\eta(\delta_0,\delta_1)$   &   $\alpha(\delta_0,\delta_1)$ & $w(\delta_0,m)$ & $w(\delta_0,m)+u(\delta_0,m)$\\
        &   $=0.0794$                   &   $=0.0567$                   &   $=0.1020$     &   $=0.0811$\\
\bottomrule   
\end{tabular}
}
\end{table}

Finally, making use of Theorem~\ref{thm:asymptotic_normality_test_statistics} we can compare the empirical power with an estimate based on the asymptotic normality of the statistic. Namely, we know the test statistics are asymptotically normal with exactly the same variance, and a small mean-shift proportional to $n^\gamma(1+\smallO(1))$, where the proportionality constant is given by $\eta(\delta_0,\delta_1)$ from \eqref{eqn:shift_T} and $\alpha(\delta_0,\delta_1)$ from \eqref{eqn:shift_Q} for $T(G_n)$ and $Q(G_n)$ respectively. Based on this, one can get an estimate for the power of the tests for different values of $\delta_1$. This is shown by the dashed lines in Figure~\ref{fig:power}. As one can see, although not terribly accurate, the estimates capture the exact behavior of the empirically observed power. This lack of accuracy is not unexpected, as all we know is that the mean-shifts are of the form $\text{const}(\delta_0,\delta_1,m) n^\gamma+\smallO(n^\gamma)$. However, the remainder term might still have an order only slightly smaller than $n^\gamma$, which will lead to rather poor power estimates for small values of $n$.

\section{Discussion and open problems} \label{sec:changepoint_detect_discussion}
In this section, we compare our results to the literature and state some open problems.

\subsection{Early changepoint} 
In previous work \cite{Bhamidi2018,Banerjee2018,Cirkovic2022}, the case of an {\em early} changepoint was considered for preferential attachment trees, i.e., for $m=1$. Thus, our work extends this setting from trees to graphs, as well as from an early changepoint to a late one. Arguably, the latter case is more relevant in practice, since one would rather detect a changepoint quickly, meaning, close to the time after which it occurs. This setting corresponds to a changepoint close to the time of observation of the final network.

\subsection{Dynamical graph observations} 
It would be of interest to extend our results to a {\em dynamic} setting, where we detect the changepoint as the graph changes dynamically. Bear in mind though that we currently only assume that we observe the graph at the final time, and observing the graph dynamically thus provides much more information. Thus, it is an interesting extension to devise an appropriate setting where we only observe {\em partial information} on the network, while still detecting the changepoint dynamically. There are several settings that could be of interest. In the first, one observes the network snapshots only at multiples of $n^\gamma$. In the second, we assume that we only dynamically observe information about the degree counts, and not the entire network. We defer such problems to future work.

\subsection{Lower bounds} Conjecture~\ref{conj:lower_bounds} states that no test will be powerful when $\gamma<\tfrac{1}{2}$. Proving such lower bounds in the context of preferential attachment models is challenging, due to the latent nature of these models. However, part (i) of the conjecture might be approached by relying on the asymptotic normality characterization in Lemma~\ref{lem:joint_CLT}, together with bounds on $\E_1[N_k(n)]-\E_0[N_k(n)]$ obtained using the methods developed in this paper.

\subsection{Boundary case \texorpdfstring{$\gamma=1$}{ɣ=1}} Note that in Theorems~\ref{thm:minimal_degree_test_known_delta} and \ref{thm:main2} the case $\gamma=1$ is excluded. This is in contrast with the results in \cite{Bhamidi2018,Banerjee2018,Cirkovic2022}. The proof of the two theorems relies on Proposition~\ref{prp:A}, that explicitly excludes the case $\gamma=1$. It should be possible to extend that result for $\gamma=1$. Specifically, the current argument quantifies the contribution to $\E_1[N_m(n)]-np_m(\delta_0)$ made by vertices that arrived after the changepoint. Due to the late changepoint, most of those vertices will have degree $m$ in $G_n$. However, for $\gamma=1$, a small, but non-vanishing fraction of those vertices will have higher degree. Therefore, to extend the result, this needs to be quantified, and a slightly more refined argument will be needed, where the role of the parameter $c$ will become much more prevalent. Extending the result of Theorem~\ref{thm:main2} to this setting will likely be significantly more challenging, as it requires extending Proposition~\ref{prp:ana-diff-means}, where the assumption that $\gamma<1$ was crucially used to bound the terms of order $\smallO(n^\gamma)$. On the other hand, when $\gamma=1$, our test statistic will likely not be a good practical choice (particularly when $c$ is large), and the statistics used in \cite{Bhamidi2018,Banerjee2018} will likely lead to more powerful tests.

\subsection{Other test statistics} Note that information about the presence of a changepoint is present not only in $N_m(n)$, but also on other counts of low-degree vertices. With this in mind, a test based on $(N_m(n),N_{m+1}(n))$ can be potentially more powerful (in a finite sample sense) than the test we proposed. Although we expect such tests to have exactly the same asymptotic performance, they can perform much better for finite $n$. An interesting avenue of research is to identify, in a principled way, statistics that lead to tests that have higher power than the ones proposed.

\subsection{Boundary case \texorpdfstring{$\gamma=\tfrac{1}{2}$}{ɣ=1/2}} When $\gamma=\tfrac{1}{2}$, the fluctuations of $N_m(n)$ around its mean under $H_0$ are of the same order as $\E_1[N_m(n)]-p_m(\delta_0)$. Moreover, since a central limit theorem holds for $N_m(n)-\E_1[N_m(n)]$ under $H_1$, with the same limiting variance as under $H_0$ (cf.~Theorem \ref{thm:asymptotic_normality_test_statistics}), it follows that, when $\gamma=\tfrac{1}{2}$, the type-II error of our test is strictly bounded away from zero. In other words, when $\gamma=\tfrac{1}{2}$, a large value of $N_m(n)-p_m$ can be explained either by a large deviation away from $\E_0[N_m]$ under $H_0$, or by a deviation around $\E_1[N_m]$ under $H_1$.

\section{Powerful test for known \texorpdfstring{$\delta_0$}{δ₀}: Proof of Theorem~\ref{thm:minimal_degree_test_known_delta}}\label{sec:proof_known}

The main idea of the proof is to decompose the test statistic $T(G_n)$ in two terms:
\begin{equation}\label{eqn:decomposition_T}
T(G_n)=N_m(n)-np_m(\delta_0)=\underbrace{\E_\ell[N_m(n)]-np_m(\delta_0)}_{\coloneqq A}+\underbrace{N_m(n)-\E_\ell[N_m(n)]}_{\coloneqq B}\ ,
\end{equation}
where $\ell\in\{0,1\}$. The characterization of the stochastic term $B$ is the same under both the null and alternative models, and follows a somewhat standard argument. Let $\ell\in\{0,1\}$ be fixed. Define the stochastic process $\{M_t\}_{t=1}^n$ such that
$$M_t\coloneqq \E_\ell[N_m(n)|G_t]\ .$$
This is a Doob martingale \cite[Lemma 8.5]{VanderHofstad2017}, and such that $\E_\ell[N_m(n)]=M_1$ and $N_m(n)=M_n$. Furthermore, by \cite[Lemma 8.6]{VanderHofstad2017}, we know that, for every $t\in\{2,\ldots,n\}$,
$|M_n-M_{n-1}|\leq 2m$ almost surely. Although strictly speaking these two lemmas were not stated for our model, their arguments do not depend on the specific sequence $\delta(t)$ in \eqref{eq:attachment_function_general}. Therefore, those results still hold and follow simply from the fact that every step in time we add precisely $m$ edges to the graph. With this in hand, we can directly apply the Azuma-Hoeffding inequality to get that, for any $x > 0$,
\begin{equation}\label{eqn:AH_Nm}
\forall \ell\in\{0,1\}\quad \P_\ell\left(\left|N_m(n) - \E_\ell[N_m(n)]\right| \geq x\right) \leq 2 \mspace{1mu} \e^{-\frac{x^2}{8 m^2 n}}\ .
\end{equation}
This completes the characterization of the term $B$ in \eqref{eqn:decomposition_T}. We now proceed by considering the type-I and type-II errors separately.

\subsection{\normalfont\emph{Type-I error}}
To control the term $A$ in \eqref{eqn:decomposition_T} under the null model we use \cite[Proposition~2.2]{Deijfen2007} (see also \cite[Proposition~8.7]{VanderHofstad2017}), which states that there exists a constant $C_0 = C_0(\delta_0, m)$ such that, for all $n \geq 1$,
\begin{equation}\label{eqn:bounded_difference}
|\E_0[N_m(n)] - n \mspace{1mu} p_m(\delta_0)| \leq C_0\ .
\end{equation}
Combining this with \eqref{eqn:AH_Nm} we see that the type-I error of the minimal degree test is bounded by
\begin{align}
\P_0(\psi(T(G_n)) \neq 0)
  &= \P_0\left(|N_m(n) - n \mspace{1mu} p_m(\delta_0)| \geq m \sqrt{8 n \log(2 / \alpha)}\right)\\
  &\leq \P_0\left(|N_m(n) - \E_0[N_m(n)]| \geq m \sqrt{8 n \log(2 / \alpha)} - C_0\right)\\
  &\leq 2 \exp\left(- \frac{(m \sqrt{8 n \log(2 / \alpha)} - C_0)^2}{8 m^2 n}\right)
  = (1 + \smallO(1)) \alpha \ .
\end{align}
This shows that the type-I error is essentially at most $\alpha$, completing the first part of the proof.

\subsection{\normalfont\emph{Type-II error}}

Again we must control the term $A$ above, for which the following proposition is instrumental:
\begin{proposition}\label{prp:A} Let $0<\gamma<1$. Then
\begin{align}
\E_1[N_m(n)]-np_m(\delta_0)&= (1+\smallO(1))\ cn^\gamma \left(1 - \frac{p_m(\delta_0)}{p_m(\delta_1)}\right)\\
 &= (1+\smallO(1)) cn^\gamma (\delta_0-\delta_1)\frac{1}{(2+\delta_1/m)(m+\delta_0+2+\delta_0/m)}\ .
\end{align}
\end{proposition}
\begin{proof} Note that
\begin{align}\label{eqn:decomposition_term_A}
\E_1[N_m(n)]-np_m(\delta_0) &= \E_0[N_m(n)]-np_m(\delta_0)\\
&\hspace{-40pt}+\E_1[N_m(n)]-\E_1[N_m(\tau_n)] - (\E_0[N_m(n)]-\E_0[N_m(\tau_n)])\ .
\end{align}
The equality holds as the law of $G_{\tau_n}$ is the same under the null and alternative models, and therefore $\E_0[N_m(\tau_n)]=\E_1[N_m(\tau_n)]$. The last two terms are controlled in a similar way.

Let $\ell\in\{0,1\}$ and note that
\begin{align}
\E_\ell[N_m(n)]-\E_\ell[N_m(\tau_n)] &=\sum_{v\in[\tau_n]} (\P_\ell(D_v(n)=m)-\P_\ell(D_v(\tau_n)=m))\\
&\qquad + \sum_{v\in[n]\setminus[\tau_n]} \P_\ell(D_v(n)=m)\label{eqn:new_vertices}\ .
\end{align}
In the above there is a contribution from vertices that arrived before the change-point (so-called ``old'' vertices), and vertices that arrived afterwards (the ``new'' vertices). The contribution by the new vertices $v\in[n]\setminus[\tau_n]$ is essentially the same regardless of the value of $\ell$, as $\P_\ell(D_v(n)=m)\approx 1$ when $v\in[n]\setminus[\tau_n]$. To see this, note that
$\P_\ell(D_v(n)=m)=1-\P_\ell(D_v(n)>m)$ and that the event $D_v(n)>m$ can only occur if there is at least one vertex $v'>v$ attaching to $v$. Referring to \eqref{eq:attachment_function_alt} the probability of this happening is at most $(m+\delta_\ell)/((2m+\delta_{\ell})\tau_n-2m)$, and there are at most $m(n-\tau_n)$ possible edges that can lead to that connection. Therefore
    $$
    \P_\ell(D_v(n)=m)\geq 1-m(n-\tau_n)\frac{m+\delta_\ell}{(2m+\delta_\ell)\tau_n-2m}=1-\bigO(n^{\gamma-1})\ .
    $$
In conclusion (since $\P_\ell(D_v(n)=m)$ is bounded above by 1)
    $$
    \sum_{v\in[n]\setminus[\tau_n]} \P_\ell(D_v(n)=m)=c(n-\tau_n)(1+\bigO(n^{\gamma-1}))\ ,\qquad \ell\in\{0,1\}\ .
    $$

For the term involving the ``old'' vertices we use the following lemma, which will be used to the full extent for the proof of Theorem~\ref{thm:main2}:
\begin{lemma}\label{lem:crucial} Let $v\in[\tau_n]$, $\gamma<1$ and $m\leq k = \smallO(n^{1-\gamma})$ and $\ell\in\{0,1\}$. Then
$$\P_\ell(D_v(n)-D_v(\tau_n)>0\mid D_v(\tau_n)=k)=(1+\smallO(1))cn^{\gamma-1} m \frac{k+\delta_\ell}{2m+\delta_\ell}\ .$$ as $n\to\infty$.
\end{lemma}
\begin{proof}
Note that
\begin{align}
\lefteqn{\P_\ell\left(D_v(n)-D_v(\tau_n)>0\mid  D_v(\tau_n)=k\right)}\\
&= 1 - \P_\ell\left(D_v(n)-D_v(\tau_n)=0\mid D_v(\tau_n)=k\right)\\
&=1-\prod_{j\in[n]\setminus[\tau_n]} \prod_{i=1}^m \left(1-\frac{k+\delta_\ell}{j(2m+\delta_\ell)-2m+i-1}\right)\\
&=1-\exp\left(\sum_{j\in[n]\setminus[\tau_n]} \sum_{i=1}^m \log\left(1-\frac{k+\delta_\ell}{j(2m+\delta_\ell)-2m+i-1}\right)\right)\\
&=1-\exp\left(\sum_{j\in[n]\setminus[\tau_n]} \sum_{i=1}^m -(1+\smallO(1))\frac{k+\delta_\ell}{j(2m+\delta_\ell)-2m+i-1}\right)\label{eqn:step1}\\
&=1-\exp\left(-(1+\smallO(1))cmn^\gamma \frac{k+\delta_\ell}{n(2m+\delta_\ell)}\right)\\
&=1-\left(1-(1+\smallO(1))cmn^{\gamma-1} \frac{k+\delta_\ell}{2m+\delta_\ell}\right)\label{eqn:step2}\\
&=(1+\smallO(1))cmn^{\gamma-1} \frac{k+\delta_\ell}{2m+\delta_\ell}\ ,
\end{align}
where in \eqref{eqn:step1} we relied on the fact that $k=\smallO(n)$ and for \eqref{eqn:step2} it is crucial that $k=\smallO(n^{1-\gamma})$.
\qed

With this lemma in hand, we clearly see that
\begin{align}
\lefteqn{\P_\ell(D_v(n)=m)-\P_\ell(D_v(\tau_n)=m)}\\
&= \P_\ell(D_v(n)-D_v(\tau_n)=0|D_v(\tau_n)=m)\P_\ell(D_v(\tau_n)=m)\\
&= \left(1-\P_\ell(D_v(n)-D_v(\tau_n)>0|D_v(\tau_n)=m)\right)\P_\ell(D_v(\tau_n)=m)\\
&= \left(1-(1+\smallO(1))cn^{\gamma-1}m\frac{m+\delta_\ell}{2m+\delta_\ell}\right)\P_\ell(D_v(\tau_n)=m)\ .
\end{align}
Putting the two results together we get
\begin{align}
\lefteqn{\E_\ell[N_m(n)]-\E_\ell[N_m(\tau_n)]}\\
&= \left(1-(1+\smallO(1))cn^{\gamma-1}m\frac{m+\delta_\ell}{2m+\delta_\ell}\right)\E_\ell[N_m(\tau_n)]+c(n-\tau_n)(1+\bigO(n^{\gamma-1}))\ .
\end{align}
Note that this results characterizes what happens both under the null and alternative models. Hence this result, together with \eqref{eqn:decomposition_term_A} and the fact that $\E_0[N_m(\tau_n)]=\E_1[N_m(\tau_n)]$, yields
\begin{align}
\lefteqn{\E_1[N_m(n)]-np_m(\delta_0)}\\
&= \E_0[N_m(n)]-np_m(\delta_0)\\
&\qquad +(1+\smallO(1))cn^{\gamma-1}m\left(\frac{m+\delta_1}{2m+\delta_1}-\frac{m+\delta_0}{2m+\delta_0}\right)\E_0[N_m(n)]+\bigO(n^{2\gamma-1})\\
&=\bigO(1)+(1+\smallO(1))cn^{\gamma-1}(\delta_1-\delta_0)\frac{1}{(2+\delta_1/m)(2+\delta_0/m)}n\left(p_m(\delta_0)+\bigO(1/n)\right)\\
&\qquad+\bigO(n^{2\gamma-1})\\
&=(1+\smallO(1))cn^{\gamma}\frac{\delta_1-\delta_0}{(2+\delta_1/m)(m+\delta_0+2+\delta_0/m)}\ ,
\end{align}
where $0<\gamma<1$, and we have again used \eqref{eqn:bounded_difference} to relate $\E_0[N_m(n)]$ to $np_m(\delta_0)$.
\end{proof}

Similarly as for the type-I error, we can use Proposition~\ref{prp:A} together with the Azuma-Hoeffding inequality \eqref{eqn:AH_Nm} to get
\begin{align}
\lefteqn{\P_1(T(G_n) \neq 1)}\\
&= \P_1\left(|N_m(n) - n p_m(\delta_0)| < m \sqrt{8 n \log(2 / \alpha)}\right)\\
&\leq \P_1\left(|N_m(n) - \E_1[N_m(n)]| > \left(|\E_1[N_m(n)]-n p_m(\delta_0)| - m \sqrt{8 n \log(2 / \alpha)}\right)\vee 0\right)\\
&\leq 2\exp\left(-\frac{\left(\left(|\E_1[N_m(n)]-n p_m(\delta_0)| - m \sqrt{8 n \log(2 / \alpha)}\right)\vee 0\right)^2}{8m^2n}\right)\ .
\end{align}
Considering the cases $\gamma > \tfrac{1}{2}$ and $\gamma = \tfrac{1}{2}$ separately, this gives
\begin{align}
\lefteqn{\P_1(T(G_n) \neq 1)}\\
  &\leq
  \begin{cases}
    \smallO(1) & \text{when } \gamma > \tfrac{1}{2},\\
    (2 + \smallO(1)) \mspace{2mu} \exp\left(-\Bigl(\Bigl(\frac{c \mspace{2mu} |1 - p_m(\delta_0) / p_m(\delta_1)|}{m \sqrt{8}} - \smash{\sqrt{\log(2 / \alpha)}}\Bigr) \vee 0\Bigr)^{\!2}\right) & \text{when } \gamma = \tfrac{1}{2}.
  \end{cases}\notag
\end{align}

The case $\gamma<\tfrac{1}{2}$ does not follow immediately from the analysis above, as the characterization obtained by the Azuma-Hoeffding only provides an upper bound on the variability of the test statistic. However, in Theorem~\ref{thm:asymptotic_normality_test_statistics} it is shown that $(T(G_n)-\E_1[T(G_n)])/\sqrt{n}$ is asymptotically normal with mean 0 and variance $w(\delta_0,m)>0$. As shown above $\E_1[T(G_n)]/\sqrt{n}=\bigO(n^{\gamma-\tfrac{1}{2}})$. Therefore $\E_1[T(G_n)]/\sqrt{n}=\smallO(1)$ when $\gamma<\tfrac{1}{2}$, and therefore $T(G_n)/\sqrt{n}$ converges to the same distribution under the null and alternative models, meaning the type II error is asymptotically just the complement of the type I error.
\end{proof}

\section{Powerful test for unknown \texorpdfstring{$\delta_0$}{δ₀}: Proof of Theorem~\ref{thm:main2}}\label{sec:proof_unknown}

To prove Theorem~\ref{thm:main2} we partially leverage on the results and analysis in Theorem~\ref{thm:minimal_degree_test_known_delta}. However, we must take into account that $\delta_0$ is not known, and rather we have only its estimate $\hat\delta_n$. The main idea is to decompose the test statistic $Q(G_n)$ as
\begin{align}\label{eqn:three_terms}
Q(G_n) &= N_m(n) - np_m(\hat \delta_n)\\
&=\underbrace{\E_\ell[N_m(n)]-np_m(\delta_0)}_{:=A}+\underbrace{np_m(\delta_0)-np_m(\tilde \delta_n)}_{:=B}\\
&\qquad + \underbrace{N_m(n)-\E_\ell[N_m(n)]+np_m(\tilde \delta_n)-np_m(\hat \delta_n)}_{:=C}\ ,
\end{align}
where $\ell\in\{0,1\}$. In the above $\tilde \delta_n$ is a deterministic quantity, and it is formally defined below. At this moment one might simply think of it as a ``population'' version of $\hat\delta_n$.

Clearly $A$ is already characterized by Proposition~\ref{prp:A}. The bulk of the argument needed to prove Theorem~\ref{thm:minimal_degree_test_known_delta} is in the characterization of $B$, the second term. For this we need to understand how fast $\tilde \delta_n$ converges to $\delta_0$ as $n\to\infty$. It turns out that under the alternative model both $A$ and $B$ have the same first-order asymptotic behavior, but with different leading constants. This fact is crucial to ensure that the test is powerful. Finally $C$, the last term, appears complicated but it is very well concentrated around zero.

Define $\tilde\delta_n$ as a solution of $\E_\ell[\iota'_n(\delta)]=0$ in $\delta\in[\delta_{\min},\delta_{\max}]$ (if more than one solution exists, then we choose an arbitrary one). Most of the analysis focuses on the alternative model, but the stated results apply to the null model simply by taking $\delta_1=\delta_0$. The following proposition gives a characterization of $\tilde\delta_n$ and the term $B$ above:

\begin{proposition}\label{prp:B} Let $\tfrac{1}{2}< \gamma<1$. Under the alternative model $\tilde \delta_n$ converges to $\delta_0$ as $n\to\infty$. Furthermore,
\begin{align}\label{eqn:convergence_tilde_delta}
\tilde\delta_n-\delta_0&=(1+\smallO(1))n^{\gamma-1} c(\delta_1-\delta_0)\frac{2m+\delta_0}{2m+\delta_1}\ ,
\end{align}
and
\begin{equation}
n(p_m(\delta_0)-p_m(\tilde\delta_n)) = (1+\smallO(1)) cn^{\gamma} \frac{1}{(m+\delta_0+2+\delta_0/m)^2}(\delta_1-\delta_0)\frac{2m+\delta_0}{2m+\delta_1}\ .
\end{equation}
\end{proposition}

The proof of this result is rather involved, due to the implicit nature of the definition of $\tilde\delta_n$. Note, however, that $\iota'_n(\delta)$ is a score function, therefore we have immediately that $\E_0[\iota'_n(\delta_0)]=0$. Under the alternative model $\tilde\delta_n$ will not be equal to $\delta_0$ and quantifying this deviation is crucial to our analysis. The following technical result, examining the difference between $\E_1[\iota'_n(\delta)]$ and $\E_0[\iota'_n(\delta)]$, is instrumental:

\begin{proposition}[Analysis of differences of means]
\label{prp:ana-diff-means}
Let $\tfrac{1}{2}<\gamma<1$. As $n\rightarrow \infty$, and uniformly for every $\delta\in[\delta_{\min},\delta_{\max}]$,
\begin{equation}
(n+1)\left(\E_1[\iota'_n(\delta)]-\E_0[\iota'_n(\delta)]\right) = \kappa(\delta_1,\delta_0,\delta) n^\gamma (1+\smallO(1))\ ,
\end{equation}
where $\kappa(\delta_1,\delta_0,\delta)$ equals
\begin{equation}\label{kappa-delta0-delta1-def}
\kappa \coloneqq \kappa(\delta_1,\delta_0,\delta)=(\delta_1-\delta_0)\frac{cm}{(2m  + \delta_1)(2m  + \delta_0)}\left(-1+\sum_{k\geq m}\frac{2m+\delta}{k+\delta}p_k(\delta_0)\right)\ .
\end{equation}
\end{proposition}

The proof of this result, which is long and rather technical, is given in the appendix. With this result in hand we are ready to prove Proposition~\ref{prp:B}:
\begin{proof}[Proof of Proposition~\ref{prp:B}]
We first establish that $\tilde\delta_n$ is well defined and that $\tilde\delta_n\to\delta_0$. Define
\begin{equation}\label{eqn:limit_iota_prime}
\iota'(\delta)\coloneqq \sum_{k\geq m} \frac{p_{>k}(\delta_0)}{k+\delta}-\frac{1}{2+\delta/m}\ ,
\end{equation}
where $p_{>k}(\delta_0)=\sum_{j>k} p_j(\delta_0)$. Intuitively, this should be the limit of $\iota'_n(\delta)$ as $n\to\infty$, as we see next. Note also that we can rewrite this expression in terms of $p_k$, by noticing (see \cite[Lemma~2]{Gao2017}) that, for $k\geq m$
\begin{equation}
p_{>k}(\delta_0)=\frac{k+\delta_0}{2+\delta_0/m}p_k(\delta_0)\ .
\end{equation}
Therefore,
$$\iota'(\delta)=\frac{1}{2+\delta_0/m}\sum_{k\geq m} \frac{k+\delta_0}{k+\delta}p_k(\delta_0) -\frac{1}{2+\delta/m} \ .$$

\cite[Lemma~6]{Gao2017} shows essentially that, as $n\rightarrow \infty$,
$$\sup_{\delta\in[\delta_{\min},\delta_{\max}]} |\iota'_n(\delta)-\iota'(\delta)| \xrightarrow{\P_0} 0\ .$$
\cite[Lemma~6]{Gao2017} is actually stated for a more general setting, where the number of edges added at each step is \textit{random}. However, if the support of the distribution of the number of edges is bounded below by $m$, then following the steps in the proof of the lemma leads to the above result. Actually, in our setting we can state a slightly stronger result, namely convergence in $L_1$. First note that both $|\iota'_n(\delta)|$ and $|\iota'(\delta)|$ are bounded, uniformly in $\delta$ and $n$. To see this note that, regardless of the value $k$,
\[
kN_{>k}(n)=k\sum_{\ell>k}N_\ell(n)\leq \sum_{\ell>k}\ell N_\ell(n)\leq \sum_{\ell\geq m}\ell N_\ell(n)=2nm\ .
\]
Therefore,
\begin{align}
0 &\leq \frac{1}{n+1}\sum_{k=m}^\infty \frac{1}{k+\delta}N_{>k}(n) \leq \frac{1}{n+1}\sum_{k=m}^\infty \frac{1}{k(k+\delta)}2nm\\
&\leq \frac{2nm}{n+1}\sum_{k=m}^\infty \frac{1}{k(k+\delta_{\min})} \leq m\sum_{k=m}^\infty \frac{1}{k(k+\delta_{\min})}\\
&\leq m\left(\frac{1}{m(m+\delta_{\min})}+\sum_{k=1}^\infty \frac{1}{k^2}\right) = m\left(\frac{1}{m(m+\delta_{\min})}+\pi^2/6\right)\ .
\end{align}
On the other hand,
\begin{align}
0&\leq \frac{1}{n+1}\sum_{t=2}^n \sum_{i=1}^m \frac{t}{t\delta+2m(t-1)+(i-1)} \leq \frac{m}{n+1}\sum_{t=2}^n \frac{t}{t\delta_{\min}+2m(t-1)}\\
&\leq \frac{m(n-1)}{n+1} \frac{2}{2(m+\delta_{\min})} \leq \frac{2m}{2(m+\delta_{\min})}\ .
\end{align}
These two results together imply that $|\iota'_n(\delta)|$ is almost surely uniformly bounded for all $\delta\in[\delta_{\min},\delta_{\max}]$. For $\iota'$ we must simply note that this is a continuous function defined in the compact set $[\delta_{\min},\delta_{\max}]$, and is therefore bounded. In conclusion, as $n\to\infty$,
\begin{align}
\E_0\left[\sup_{\delta\in[\delta_{\min},\delta_{\max}]} |\iota'_n(\delta)-\iota'(\delta)|\right]\to 0\ .
\end{align}
This result, together with Proposition~\ref{prp:ana-diff-means}, shows that this is also true under the alternative model, and therefore, as $n\to\infty$,
\begin{equation}\label{eqn:L1_convergence}
\forall \ell\in\{0,1\}\quad \E_\ell\left[\sup_{\delta\in[\delta_{\min},\delta_{\max}]} |\iota'_n(\delta)-\iota'(\delta)|\right]\to 0\ .
\end{equation}

By definition $\E_1[\iota'_n(\tilde\delta_n)]=0$, therefore we conclude that $\iota'(\tilde\delta_n)\to 0$. \cite[Lemma~4]{Gao2017} shows that $\iota'$ has a unique zero at $\delta_0$, and $\iota'(\delta)>0$ for $\delta<\delta_0$ and $\iota'(\delta)<0$ for $\delta>\delta_0$. This immediately implies that $\tilde\delta_n\to\delta_0$ as $n\to\infty$, proving the first assertion in the proposition.

To quantify the speed of convergence note first that $\E_0[\iota'_n(\delta_0)]=0$, since $\iota'_n$ is a score function. Furthermore, by definition $\E_1[\iota'_n(\tilde\delta_n)]=0$, therefore Proposition~\ref{prp:ana-diff-means} implies that
\begin{align}
0&=\E_1[\iota'_n(\tilde\delta_n)]-\E_0[\iota'_n(\tilde\delta_n)]+\E_0[\iota'_n(\tilde\delta_n)]\\
&=\kappa(\delta_1,\delta_0,\tilde\delta_n) n^{\gamma-1} (1+\smallO(1))+\E_0[\iota'_n(\delta_0)+\iota''_n(\bar\delta_n)(\tilde\delta_n-\delta_0)]\\
&=\kappa(\delta_1,\delta_0,\tilde\delta_n) n^{\gamma-1} (1+\smallO(1))+\E_0[\iota''_n(\bar\delta_n)](\tilde\delta_n-\delta_0)\ ,\label{eqn:pop_est}
\end{align}
where $|\bar\delta_n-\delta_0|\leq |\tilde\delta_n-\delta_0|$. Clearly $\bar\delta_n\to\delta_0$. As shown in \cite{Gao2017}, $\iota''_n(\delta)$ also converges in probability to $\iota''(\delta)$ uniformly in $\delta\in[\delta_{\min},\delta_{\max}]$. In fact, convergence holds also in $L^1$, since $\iota''_n(\delta)$ is uniformly bounded (following the type of argument used before showing that $\iota'_n$ is uniformly bounded). In addition, it is also shown in \cite{Gao2017} that $\iota''(\delta_0)<0$. Therefore, for $n$ large enough,
$$\E_0[\iota''_n(\bar\delta_n)]=(1+\smallO(1))\iota''(\delta_0)<0\ .$$

We are now ready to show the second assertion in Proposition~\ref{prp:B}. Note that $\kappa(\delta_1,\delta_0,\tilde\delta_n)\to \kappa(\delta_1,\delta_0,\delta_0)$ since $\kappa$ is a continuous function. Putting all this together and re-writing the expression \eqref{eqn:pop_est} we conclude that
\begin{align}
\tilde\delta_n-\delta_0&=(1+\smallO(1))n^{\gamma-1} \frac{\kappa(\delta_1,\delta_0,\delta_0)}{|\iota''(\delta_0)|}\ ,
\end{align}
where
\begin{align}
\iota''(\delta_0)&=\frac{m}{(2m+\delta_0)^2}-\sum_{k\geq m}\frac{p_{>k}(\delta_0)}{(k+\delta_0)^2}\\
&=\frac{m}{(2m+\delta_0)^2}-\frac{m}{2m+\delta_0}\sum_{k\geq m}\frac{p_k(\delta_0)}{k+\delta_0}\ .
\end{align}
Therefore, after trivial algebraic manipulation, we conclude that
\begin{align}
\tilde\delta_n-\delta_0&=(1+\smallO(1))n^{\gamma-1} \frac{c(\delta_1-\delta_0)}{2m+\delta_1}\frac{-1+\sum_{k\geq m}\frac{2m+\delta_0}{k+\delta_0}p_k(\delta_0)}{{\sum_{k\geq m} \frac{1}{k+\delta_0}p_k(\delta_0)}-\frac{1}{2m+\delta_0}}\\
&=(1+\smallO(1))n^{\gamma-1} c(\delta_1-\delta_0)\frac{2m+\delta_0}{2m+\delta_1}\ ,
\end{align}
proving the second assertion in the proposition. For the last assertion note that $\delta\mapsto p_m(\delta)$ is a continuously differentiable function, so that
$$p_m(\tilde\delta_n)=p_m(\delta_0)+p'_m(\bar\delta_n)(\tilde\delta_n-\delta_0)\ ,$$
where $|\bar\delta_n-\delta_0|\leq |\tilde\delta_n-\delta_0|$ and
\begin{equation}\label{eqn:p_m_prime}
p'_m(\delta)=-\frac{1}{(m+\delta+2+\delta/m)^2}\ .
\end{equation}
Clearly $\bar\delta_n\to \delta_0$, and $p'_m(\delta_0)\neq 0$ and so
\begin{align}
\lefteqn{n(p_m(\delta_0)-p_m(\tilde\delta_n))}\\
&= -(1+\smallO(1)) np'_m(\delta_0)(\tilde\delta_n-\delta_0)\\
&= (1+\smallO(1)) n^{\gamma} \frac{1}{(m+\delta_0+2+\delta_0/m)^2}c(\delta_1-\delta_0)\frac{2m+\delta_0}{2m+\delta_1}\ .
\end{align}
\end{proof}

In conclusion, the sum of the terms $A$ and $B$ in \eqref{eqn:three_terms} equals $\alpha(\delta_0,\delta_1)(1+\smallO(1))n^{\gamma}$ with 
\begin{align}
\alpha(\delta_0,\delta_1)&=c \left[1-\frac{p_m(\delta_0)}{p_m(\delta_1)}+\frac{1}{(m+\delta_0+2+\delta_0/m)^2}(\delta_1-\delta_0)\frac{2m+\delta_0}{2m+\delta_1}\right]\\
&=c(\delta_0-\delta_1)\frac{m+\delta_0}{(2+\delta_1/m)(m+\delta_0+2+\delta_0/m)^2}\ .
\end{align}
Clearly, this takes the value $0$ when $\delta_0=\delta_1$, and it is non-zero otherwise.

To complete the proof of Theorem~\ref{thm:minimal_degree_test_known_delta} we need to characterize the term $C$ in \eqref{eqn:three_terms}. This has two components, the first one already studied in the proof of Theorem~\ref{thm:minimal_degree_test_known_delta}. For the second component we must characterize the deviations of $\hat\delta_n$ around $\tilde\delta_n$. Again, due to the implicit definition of the estimator this requires some care. The results are summarized in the following proposition, proven in the appendix:
\begin{proposition}\label{prp:C2}
The estimator $\hat\delta_n$ is consistent, both under the null and alternative models. Specifically,
$$\E_\ell[|\hat\delta_n-\delta_0|]\to 0\ ,$$
as $n\to\infty$, where $\ell\in\{0,1\}$. In addition,
$$|\hat\delta_n-\tilde\delta_n|=\smallO_{\P_\ell}\left(a_n/n\right)$$
where $a_n=\omega(\sqrt{n}\log n)$. Finally,
$$n(p_m(\hat\delta_n)-p_m(\tilde\delta_n))=\smallO_{\P_\ell}(a_n)\ .$$
\end{proposition}

With these results in hand we are ready to complete the proof of Theorem~\ref{thm:main2}. We proceed separately for the type I and type II error:

\subsection{\normalfont\emph{Type-I error}}
Recall the definition of $a_n$ in Theorem~\ref{thm:main2}, which satisfies $a_n=\omega(\sqrt{n}\log n)$ and $a_n=n^{\tfrac{1}{2}+o(1)}$. Refer to the decomposition of the test statistic $Q(G_n)$ in \eqref{eqn:three_terms}. Under the null model the term $B$ is equal to zero, and for the term $A$ we know (from \eqref{eqn:bounded_difference}) that
$$|\E_0[N_m(n)] - n \mspace{1mu} p_m(\delta_0)| \leq C_0(\delta_0,m)\coloneqq C_0\ .$$
Therefore
\begin{align}
\lefteqn{\P_0\left(|Q(G_n)|\geq a_n\right)}\\
&= \P_0\left(\left|\E_0[N_m(n)]-np_m(\delta_0)+N_m(n)-\E_0[N_m(n)]+np_m(\tilde \delta_n)-np_m(\hat \delta_n)\right|\geq a_n\right)\\
&\leq \P_0\left(\left|N_m(n)-\E_0[N_m(n)]+np_m(\tilde \delta_n)-np_m(\hat \delta_n)\right|\geq a_n-C_0\right)\\
&\leq \P_0\left(\left|N_m(n)-\E_0[N_m(n)]\right|\geq \frac{a_n-C_0}{2}\right) + \P_0\left(\left|np_m(\tilde \delta_n)-np_m(\hat \delta_n)\right|\geq \frac{a_n-C_0}{2}\right)\\
&=\smallO(1)\ ,
\end{align}
where the final result follows from \eqref{eqn:AH_Nm} and the last statement in Proposition~\ref{prp:C2}, since $a_n=\omega(\sqrt{n}\log n)$.

\subsection{\normalfont\emph{Type-II error}}

Referring again to the decomposition in \eqref{eqn:three_terms} we see that the terms $A$ and $B$ now play a crucial role. Namely, we have shown that the sum of the two terms equals
$\alpha(\delta_0,\delta_1)n^{\gamma}+\smallO(n^{\gamma})$, where 
\begin{align}
\alpha(\delta_0,\delta_1)&=c \left[1-\frac{p_m(\delta_0)}{p_m(\delta_1)}+\frac{1}{(m+\delta_0+2+\delta_0/m)^2}(\delta_1-\delta_0)\frac{2m+\delta_0}{2m+\delta_1}\right]\\
&=c(\delta_0-\delta_1)\frac{m+\delta_0}{(2+\delta_1/m)(m+\delta_0+2+\delta_0/m)^2}\label{eqn:shift_Q}\ .
\end{align}
Clearly, this takes the value $0$ when $\delta_0=\delta_1$, and it is non-zero otherwise. The characterization of the term $C$ remains exactly the same as under the null model. Therefore,
$$\E_\ell[N_m(n)]-np_m(\delta_0)+np_m(\delta_0)-np_m(\tilde \delta_n)=\alpha(\delta_0,\delta_1)(1+\smallO(1))n^{\gamma}=\omega(a_n)\ ,$$
since $a_n=n^{\tfrac{1}{2}+\smallO(1)}$ and $\gamma>\tfrac{1}{2}$, and therefore
$$\P_1(|Q(G_n)|\geq \tau_n)=1$$
concluding the proof of Theorem~\ref{thm:main2}.
\qed

\section{Asymptotic normality proofs}\label{sec:proof_asymptotic_normality_test_statistics}

This section is dedicated to the proof of Theorem~\ref{thm:asymptotic_normality_test_statistics}. The proof for the null model relies on a somewhat involved application of a martingale central limit theorem to an appropriately constructed martingale. The proof for the alternative model hinges on the null model result, as well as further estimates of the effect of the last $n-\tau_n$ vertices on the distribution of $N_k(n)$ for $k\geq m$. Since one proof hinges on the other, and to avoid any confusion between the null and alternative models, we separate these two cases in two sections. Specifically, in Section~\ref{sec:proof_asymptotic_normality_null_model} we prove the asymptotic normality of test statistics under the null model, i.e., \eqref{eq:asymptotic_normality_unknown_delta_test_statistic}. In Section~\ref{sec:proof_asymptotic_normality_alternative_model} we prove the asymptotic normality of test statistics under the alternative model, i.e., \eqref{eq:asymptotic_normality_known_delta_test_statistic_alternative_hypothesis} and \eqref{eq:asymptotic_normality_unknown_delta_test_statistic_alternative_hypothesis}.

\subsection{Asymptotic normality under the null hypothesis}\label{sec:proof_asymptotic_normality_null_model}
Let $\mathcal{N}(\mu,\Sigma)$ denote the normal distribution with mean $\mu$ and covariance matrix $\Sigma$. The following result establishes the joint asymptotic normality of $N_m(n)$ and $\hat\delta_n$, which can be them used to deduce \eqref{eq:asymptotic_normality_unknown_delta_test_statistic} by an application of the delta method:

\begin{proposition}[Joint normality of count of degree $m$ vertices and estimator for $\delta_0$]\label{prop:joint_asymptotic_normality_degree_counts_estimator_delta_zero}
Under the null model, as $n\to\infty$, 
\begin{align}\label{eq:asymptotic_normality_known_delta_test_statistic}
\sqrt{n}\begin{pmatrix}
\frac{N_m(n)-np_m(\delta_0)}{n} \\
\hat\delta_n-\delta_0
\end{pmatrix}
\xrightarrow{D} \mathcal{N}\left(\mspace{-3mu}\begin{pmatrix} 0\\0 \end{pmatrix},\Sigma(\delta_0,m)\right)\ ,
\end{align}
where the covariance matrix is
\begin{align}\label{eq:asymptotic_normality_covariance_matrix}
\Sigma(\delta_0,m) \coloneqq 
\begin{pmatrix}
w(\delta_0,m) & -\frac{b(\delta_0,m)}{v(\delta_0,m)}\\
-\frac{b(\delta_0,m)}{v(\delta_0,m)} & \frac{1}{v(\delta_0,m)}
\end{pmatrix}\ ,
\end{align}
where $w$ and $v$ are defined in \eqref{w-var-def} and \eqref{v-var-def}, and $b(\delta_0,m)\coloneqq \frac{m^2}{(\delta_0 + m(2 + m + \delta_0))^2}$.
\end{proposition}%
\begin{proof}[Proof of Proposition~\ref{prop:joint_asymptotic_normality_degree_counts_estimator_delta_zero}]
The convergence of the marginals in \eqref{eq:asymptotic_normality_known_delta_test_statistic} follows from results in the literature, which we briefly recall here. From the general result in \cite{Baldassarri2021} it follows that, under the null model,
\begin{align}%
\frac{N_m(n)-np_m(\delta_0)}{\sqrt{n}}\ \xrightarrow{D} \mathcal{N}\left(0,w(\delta_0,m)\right)\ ,
\end{align}%
as $n\to\infty$. Moreover, from \cite{Gao2017} follows that, under the null model,
\begin{align}
\sqrt{n}(\hat\delta_n-\delta_0)\ \xrightarrow{D} \mathcal{N}\left(0,1/v(\delta_0,m)\right)\ ,
\end{align}
as $n\to\infty$. To show joint convergence, we apply the multivariate martingale central limit theorem (MMCLT) to (a linear transformation of) the vector on the left-hand side of \eqref{eq:asymptotic_normality_known_delta_test_statistic}. For more details on the MMCLT, see, e.g., \cite{crimaldi2005convergence} and references therein. In fact, the asymptotic normality of the marginals has been proven respectively in \cite{Baldassarri2021} and \cite{Gao2017} by applying a martingale central limit theorem to appropriately constructed martingale difference sequences $X_{t,i}$ and $Y_{t,i}$, which we define precisely later on. To obtain our results, we apply the MMCLT to the (square-integrable) martingale difference array $(X_{t,i}, Y_{t,i})$. The MMCLT requires two ingredients. First, the components of the array must satisfy the Lindeberg condition. Since this is also an ingredient for the univariate martingale central limit theorem, it has already been verified for $X_{t,i}$ in \cite[Section 3.2]{Baldassarri2021}  and for  $Y_{t,i}$ in \cite[Section 4]{Gao2017}. Second, one should compute the asymptotic expression for the covariance matrix. Note that the asymptotic variance of $X_{t,i}$ (resp.~$Y_{t,i}$) has already been computed explicitly in \cite{Baldassarri2021} (resp.~\cite{Gao2017}).

We begin by introducing the notation required to define the martingale difference array $(X_{t,i}, Y_{t,i})$, which, roughly speaking, is a linear transformation of $(N_m(t), \hat\delta_t)$. We then compute the asymptotic covariance matrix of $(X_{t,i}, Y_{t,i})$. Finally, we perform a linear transformation to deduce a joint central limit theorem for $(N_m(n), \hat\delta_n)$.

To construct the martingale difference sequences, it is important to consider the ``intermediate'' $m$ steps in the preferential attachment process, at each time $t$. For $i=1,\ldots,m-1$, we denote by $N_k(s,i)$ the number of vertices of degree $k$ after the $i$-th edge has been attached at time $s$, excluding the vertex $s$. We set $N_k(s+1,0)\coloneqq N_k(s,m) = N_k(s)$. We further define $D_{s,i}\coloneqq D_{v_{s,i}}(G_{s,i-1})$. In other words, $D_{s,i}$ is the degree of the vertex which will be attached to when constructing $G_{s,i}$ from $G_{s,i-1}$. For compactness, we define the constants
\begin{align}%
W_{t,i} \coloneqq \prod_{s=1}^{t-1} \prod_{j=1}^{m} a_{s,j} \prod_{j=1}^{i} a_{t,j}\qquad \text{where }\qquad a_{t,i} \coloneqq 1 - \frac{m+\delta_0}{S_{t,i-1}} \ ,
\end{align}%
and where
\begin{align}%
S_{t,i} &\coloneqq t \delta_0 + 2(t-1)m + i\ .
\end{align}%
Define also
\begin{align}
\xi &\coloneqq \frac{m+\delta_0}{2m + \delta_0}\ ,\\
\Lambda &\coloneqq \prod_{j=0}^{m-1} \frac{\Gamma\bigl(1-\frac{2m-j}{2m+\delta_0}\bigr)}{\Gamma\bigl(1-\frac{3m-j+\delta_0}{2m+\delta_0}\bigr)}\ .
\end{align}%
In the constants defined above, we omitted the dependence on $\delta_0$ and $m$ to facilitate readability of the computations that follow. The martingale difference sequences (from \cite{Baldassarri2021,Gao2017a}) are defined as
\begin{align}
X_{t,i} &\coloneqq \left(\frac{N_m(t,i) - a_{t,i-1} N_m(t,i-1) + \mathds 1{\{i = m\}}}{W_{t,i-1}}\right) \frac{1}{n^{\tfrac{1}{2} + m \xi}}\ ,\\
Y_{t,i} &\coloneqq \left(\frac{1}{D_{t,i} + \delta_0} - \frac{t}{S_{t,i-1}}\right) \frac{1}{n^{\tfrac{1}{2}}}\ .
\end{align}%
Here $X_{t,i}$ is the first component of the martingale difference array in \cite{Baldassarri2021}, and $Y_{t,i}$ is the martingale difference defined in \cite{Gao2017}. Next, we compute the asymptotic covariance matrix. By \cite[Section 3.1]{Baldassarri2021} and \cite[Section 4]{Gao2017}, as $n\rightarrow \infty$,
\begin{align}
\sum_{t=2}^n \sum_{i=1}^m \E[X_{t,i}^2 | G_{t,i-1}]
  &\xrightarrow{\P_0} \frac{1}{\Lambda^2} \frac{m}{1 + 2 m \xi} p_m \xi \bigl(1 - p_m \xi\bigr)\\
  &= \frac{1}{\Lambda^2} \frac{m^2 (m+\delta_0) (1 + m + \delta_0) (2m+\delta_0)}{(\delta_0 + 2m(1+m+\delta_0)) (\delta_0 + m(2+m+\delta_0))^2} \ , \label{eq:asymptotic_variance_X}\\
\sum_{t=2}^n \sum_{i=1}^m \E[Y_{t,i}^2 | G_{t,i-1}]
  &\xrightarrow{\P_0} \sum_{k=1}^\infty \frac{m p_k}{(k+\delta_0)(2m+\delta_0)} - \frac{m}{(2m+\delta_0)^2}\label{eq:asymptotic_variance_Y}\ .
\end{align}
We are left to explicitly compute $\E[X_{t,i} Y_{t,i} | G_{t,i-1}]$. First, since $\E[Y_{t,i} | G_{t,i-1}] = 0$, 
\begin{align}
\E[X_{t,i} Y_{t,i} | G_{t,i-1}]
&= \E\left[\frac{1}{n^{\tfrac{1}{2} + m \xi}} \frac{N_m(t,i)}{W_{t,i-1}} Y_{t,i} \middle| G_{t,i-1}\right]\\
&= \frac{1}{n^{1 + m \xi}} \frac{1}{W_{t,i-1}} \E\left[\frac{N_m(t,i)}{D_{t,i} + \delta_0} \middle| G_{t,i-1}\right] \\
&\quad- \frac{1}{n^{1 + m \xi}} \frac{1}{W_{t,i-1}} \frac{t}{S_{t,i-1}} \E\left[N_m(t,i) \middle| G_{t,i-1}\right]\ .
\end{align}
To continue, we determine the above two (conditional) expectations. We start by computing
\begin{align}
\E\left[\frac{N_m(t,i)}{D_{t,i} + \delta_0} \middle| G_{t,i-1}\right]
&= \frac{N_m(t,i-1) - 1 + \mathds 1{\{i = m\}}}{m + \delta_0} \P(D_{t,i} = m | G_{t,i-1})\\
&\quad{} + \sum_{k=m+1}^\infty \frac{N_m(t,i-1) + \mathds 1{\{i = m\}}}{k + \delta_0} \P(D_{t,i} = k | G_{t,i-1}) \notag\\
&= \sum_{k=m}^\infty \frac{N_m(t,i-1) + \mathds 1{\{i = m\}}}{k + \delta_0} \P(D_{t,i} = k | G_{t,i-1})\\
&\quad- \frac{1}{m + \delta_0} \P(D_{t,i} = m | G_{t,i-1})\\
&= \frac{t}{S_{t,i-1}} \bigl(N_m(t,i-1) + \mathds 1{\{i = m\}}\bigr) - \frac{N_m(t,i-1)}{S_{t,i-1}} \,.
\end{align}
Also, from \cite{Baldassarri2021},
\begin{equation}
\E\left[N_m(t,i) \middle| G_{t,i-1}\right]
  = \left(1 - \frac{m+\delta_0}{S_{t,i-1}}\right) N_m(t,i-1) + \mathds 1{\{i = m\}}
 \,.
\end{equation}
Combining the above, we get
\begin{align}
\E[X_{t,i} Y_{t,i} | G_{t,i-1}]
&= \frac{1}{n^{1 + m \xi}} \frac{1}{W_{t,i-1}} \E\left[\frac{N_m(t,i)}{D_{t,i} + \delta_0} \middle| G_{t,i-1}\right] \\
&\quad- \frac{1}{n^{1 + m \xi}} \frac{1}{W_{t,i-1}} \frac{t}{S_{t,i-1}} \E\left[N_m(t,i) \middle| G_{t,i-1}\right]\\
&= \frac{1}{n^{1 + m \xi}} \frac{1}{W_{t,i-1}} \frac{t}{S_{t,i-1}} \bigg[N_m(t,i-1) + \mathds 1{\{i = m\}} - \frac{N_m(t,i-1)}{t}\\
&\hspace{90pt}{} - \left(1 - \frac{m+\delta_0}{S_{t,i-1}}\right) N_m(t,i-1) - \mathds 1{\{i = m\}}\bigg] \notag\\
&= \frac{1}{n^{1 + m \xi}} \frac{N_m(t,i-1)}{W_{t,i-1} S_{t,i-1}} \left[\frac{t}{S_{t,i-1}}(m+\delta_0) - 1\right] \,.
\end{align}
To compute the limiting correlation we use that, uniformly over $i = 1, \ldots, m$, as $t\to\infty$,
\begin{equation}
\frac{N_m(t,i)}{t} \xrightarrow{\P_0} p_m\,,
\qquad
\frac{S_{t,i}}{t} \to 2 + \delta_0\ .
\end{equation}
Furthermore, tedious but straightforward computations show that, uniformly over $i = 1, \ldots, m$, as $t\to\infty$,
\begin{equation}
W_{t,i} \asymp \frac{\Lambda}{t^{m \xi}}\ ,
\end{equation}
where, crucially, $\Lambda$ is not a function of $t$. Hence, as $t\to\infty$,
\begin{equation}
\E[X_{t,i} Y_{t,i} | G_{t,i-1}]
  = \frac{1}{n^{1 + m \xi}} \frac{t^{m \xi}}{\Lambda} \frac{p_m}{(2m + \delta_0)} (\xi - 1)(1+\smallOp(1)) \,.
\end{equation}
Summing these terms and using $\sum_{t=1}^n t^{x-1} \asymp n^{x} / x$ as $n \to \infty$, gives
\begin{align}\label{eq:asymptotic_covariance_final_result}
\sum_{t=2}^n \sum_{i=1}^m \E[X_{t,i} Y_{t,i} | G_{t,i-1}]
  &\xrightarrow{\P_0} \frac{1}{\Lambda} \frac{m p_m}{(2m + \delta_0)} \frac{\xi - 1}{1 + m\xi}\\
&\quad= - \frac{m^2}{(\delta_0 + m(2 + m + \delta_0))^2} \frac{1}{\Lambda}\ .
\end{align}
In conclusion, putting together \eqref{eq:asymptotic_variance_X}, \eqref{eq:asymptotic_variance_Y}, and \eqref{eq:asymptotic_covariance_final_result} and applying the MMCLT (see, e.g., \cite{crimaldi2005convergence}) gives, as $n\to\infty$,
\begin{equation}
\sqrt{n} \left(\mspace{-3mu}
\begin{pmatrix}
\frac{N_m(n)/n}{\Lambda} \\ \iota_n'(\delta_0)
\end{pmatrix}
-
\begin{pmatrix}
\frac{p_m}{\Lambda} \\ \iota'(\delta_0)
\end{pmatrix}\mspace{-3mu}
\right)
\xrightarrow{D}\mathcal{N}\left(\mspace{-3mu}\begin{pmatrix} 0\\0 \end{pmatrix},\tilde{\Sigma}\right)\ ,\label{eq:asymptotic_normality_proof_intermediate_result}
\end{equation}
where
\begin{equation}
\widetilde{\Sigma} \coloneqq
\begin{pmatrix}
\frac{1}{\Lambda^2} \frac{m^2 (m+\delta_0) (1 + m + \delta_0) (2m+\delta_0)}{(\delta_0 + 2m(1+m+\delta_0)) (\delta_0 + m(2+m+\delta_0))^2}
  & -\frac{1}{\Lambda} \frac{m^2}{(\delta_0 + m(2 + m + \delta_0))^2}\\
-\frac{1}{\Lambda} \frac{m^2}{(\delta_0 + m(2 + m + \delta_0))^2}
  & v(\delta_0,m)
\end{pmatrix}\ .
\end{equation}
Recall that $v(\delta_0,m)$ is given in \eqref{v-var-def}. Note that the above central limit theorem result is not for $(N_m(n) / n, \hat{\delta}_n)$. This is because the martingale difference sequences are defined as a linear transformation of $(N_m(n) / n, \hat{\delta}_n)$. Therefore, next we apply a linear transformation to \eqref{eq:asymptotic_normality_proof_intermediate_result} to obtain a central limit theorem for $(N_m(n) / n, \hat{\delta}_n)$. To this end, observe that (see \cite[proof of Theorem~2]{Gao2017}),
\begin{equation}
\sqrt{n}(\hat{\delta}_n - \delta_0)
  = -\sqrt{n}\frac{(\iota_n'(\delta_0) - \iota'(\delta_0))}{\iota_n''(\bar\delta_n)} \,,
\end{equation}
where $\bar\delta_n$ lies between $\delta_0$ and $\hat{\delta}_n$, and so $\iota_n''(\bar\delta_n) \xrightarrow{\P_0} \iota''(\delta_0) = -v(\delta_0,m)$.
Using the above we obtain
\begin{align}%
\sqrt{n}(N_m(n)/n - p_m) &= \sqrt{n}\Lambda\left(\frac{N_m(n)/n}{\Lambda} -\frac{p_m}{\Lambda}\right)\\
\sqrt{n}(\hat{\delta}_n-\delta_0) &= \frac{\sqrt{n}}{v(\delta_0,m)}(\iota_n'(\delta_0) - \iota'(\delta_0))\ .
\end{align}%
In conclusion, using Slutsky's lemma, as $n\to\infty$,
\begin{align}%
\sqrt{n} \left(\mspace{-3mu}
\begin{pmatrix}
N_m(n)/n \\ \hat{\delta}_n
\end{pmatrix}
-
\begin{pmatrix}
p_m \\ \delta_0
\end{pmatrix}\mspace{-3mu}
\right)
\xrightarrow{D}
\mathcal{N}\left(\mspace{-3mu}\begin{pmatrix} 0\\0 \end{pmatrix},\Sigma(\delta_0,m)\right)\ ,
\end{align}%
where
\begin{equation}
\Sigma(\delta_0,m) \coloneqq 
\begin{pmatrix}
\frac{m^2 (m+\delta) (1 + m + \delta) (2m+\delta)}{(\delta + 2m(1+m+\delta)) (\delta + m(2+m+\delta))^2}
  & -\frac{m^2}{(\delta + m(2 + m + \delta))^2} \frac{1}{v(\delta_0,m)}\\
-\frac{m^2}{(\delta + m(2 + m + \delta))^2} \frac{1}{v(\delta_0,m)}
  &  \frac{1}{v(\delta_0,m)}
\end{pmatrix}\ .
\end{equation}

\end{proof}

The proof of \eqref{eq:asymptotic_normality_unknown_delta_test_statistic} is now a simple consequence of Proposition~\ref{prop:joint_asymptotic_normality_degree_counts_estimator_delta_zero} together with the delta method. Define the function $h(x, y) \coloneqq x - p_m(y)$. The gradient of $h(\cdot, \cdot)$ is given by
\begin{equation}
\nabla h(x, y) = (1, -p_m'(y)) = \Bigl(
1,\; \frac{m^2}{(y + m(2 + m + y))^2}
\Bigr),
\end{equation}
where $p_m'(y)$ is given in \eqref{eqn:p_m_prime}. Using $h(\cdot,\cdot)$, we can rewrite
\begin{equation}
\label{eq:unknown_delta_test_statistic2}
\frac{Q(G_n)}{\sqrt n} = \frac{N_m(n) - n \mspace{1mu} p_m(\hat{\delta}_n)}{\sqrt{n}}
  = \sqrt{n} \mspace{2mu} h\bigl(N_m(G_n)/n, \hat{\delta}_n\bigr)\,.
\end{equation}
Hence, by applying the delta method we get that \eqref{eq:unknown_delta_test_statistic2} converges in distribution as $n\to\infty$ to a normal distribution with variance 
\begin{align}
\label{eq:asymptotic_variance}
&\lim_{n \to \infty} \mathrm{Var}\Big(\frac{N_m(G_n) - n \mspace{1mu} p_m(\hat{\delta}_n)}{\sqrt{n}}\Big)= \bigl(\nabla h(p_m(\delta_0), \delta_0)\bigr)^{\! T} \:  \Sigma(\delta_0,m)\; \bigl(\nabla h(p_m(\delta_0), \delta_0)\bigr)\\
&\quad= \frac{m^2 (m+\delta_0) (1 + m + \delta_0) (2m + \delta_0)}{(\delta_0 + 2m(1 + m + \delta_0)) (\delta_0 + m(2 + m + \delta_0))^2} - \frac{m^4}{v(\delta_0,m) (\delta_0 + m(2 + m + \delta_0))^4} \,. \notag
\end{align}
where $\Sigma(\delta_0,m)$ and $v(\delta_0,m)$ are given by \eqref{eq:asymptotic_normality_covariance_matrix} and \eqref{v-var-def}, respectively.
\qed

\subsection{Asymptotic normality under the alternative hypothesis}
\label{sec:proof_asymptotic_normality_alternative_model}
The main insight is that most of the contribution for the asymptotic distribution of the degree counts is due to the attachment process up to the changepoint. In fact, the asymptotic distribution of $(N_m(n),N_{m+1}(n),\ldots)$ is normal, both under the null and alternative models, with exactly the same covariance structure but with different means. Specifically, in \cite{Baldassarri2021} it was shown that under the null model
$$\left(\frac{N_k(n)-\E_0[N_k(n)]}{\sqrt{n}}\right)_{k\geq m} \xRightarrow{D} \left(Z_k\right)_{k\geq m}\ ,$$
as $n\to\infty$, where the right-hand side is a zero-mean Gaussian process with covariance given in \cite[Theorem 2.5]{Baldassarri2021}, and the notation $(X_k(n))_{k\geq m} \xRightarrow{\smash{\raisebox{-1.0pt}{$\scriptstyle{}D$}}} (Z_k)_{k\geq m}$ means that for any $k\geq m$ we have $(X_m(n),\ldots,X_k(n)) \xrightarrow{\smash{\raisebox{-1.5pt}{$\scriptstyle{}D$}}} (Z_m,\ldots,Z_k)$ as $n\to\infty$ (i.e., the infinite vector converges in the product topology). The following lemma generalizes this result to our alternative model, immediately implying \eqref{eq:asymptotic_normality_known_delta_test_statistic_alternative_hypothesis}. Furthermore, it provides a stepping stone towards the proof of \eqref{eq:asymptotic_normality_unknown_delta_test_statistic_alternative_hypothesis}:

\begin{lemma}\label{lem:joint_CLT}
Under the alternative model with $\gamma\in(0,1)$
$$\left(\frac{N_k(n)-\E_1[N_k(n)]}{\sqrt{n}}\right)_{k\geq m} \xRightarrow{D} \left(Z_k\right)_{k\geq m}\ ,$$
as $n\to\infty$ where the right-hand side is a zero-mean Gaussian process with covariance given in \cite[Theorem 2.5]{Baldassarri2021}.
\end{lemma}
\begin{proof}
Let $k\geq m$ be arbitrary. We must show that the standardized version of $\left(N_m(n),\ldots,N_k(n)\right)$ is asymptotically normal with the correct covariance structure. Let us first inspect the asymptotic marginal distributions of $N_k(n)$, as this illustrates the main premise of the argument, that can be extended easily by using an application of the Cram\'{e}r-Wold device. Proceed by using the following decomposition
\begin{align}
\frac{1}{\sqrt{n}}\left(N_k(n)-\E_1[N_k(n)]\right)&=\frac{1}{\sqrt{n}}\left(\E_1\left[\left.N_k(n)\right|G_{\tau_n}\right]-\E_1\left[N_k(n)\right]\right)\\
&\qquad +\frac{1}{\sqrt{n}}\left(N_k(n)-\E_1\left[\left.N_k(n)\right|G_{\tau_n}\right]\right)\ .\label{eqn:decomposition}
\end{align}
The second term in \eqref{eqn:decomposition} can be dealt with conveniently using the Azuma-Hoeffding inequality. Define the Doob martingale, for $t\in [n]\setminus [\tau_n]$,
$$M_k(t)=\E_1\left[N_k(n)\mid G_t\right]\ .$$
As used before we know that almost surely $|M_k(t)-M_k(t-1)|\leq 2m$. Therefore, by the Azuma-Hoeffding inequality,
$$\P_1\left(|N_k(n)-\E_1\left[\left.N_k(n)\right|G_{\tau_n}\right]|\geq x \right)\leq 2\e^{-\frac{x^2}{8m^2 (n-\tau_n)}}\ .$$
Thus, the second term in \eqref{eqn:decomposition} is $\bigO_{\P_1}(\sqrt{n-\tau_n}/\sqrt{n})=\bigO_{\P_1}(n^{(\gamma-1)/2})=\smallO(1)$ when $\gamma<1$.

We now shift the focus to the first term in \eqref{eqn:decomposition}. Define
$$N_k(\tau_n,n)=\sum_{v\in[\tau_n]} \mathds 1\{D_v(n)=k\}\ .$$
In words, this is the number of vertices of degree $k$ that were added to the graph $G_n$ up to the changepoint. Note that
$$N_k(n)=N_k(\tau_n,n)+\sum_{v\in[n]\setminus[\tau_n]} \mathds 1\{D_v(n)=k\}\ .$$
Given the affine nature of the preferential attachment function (after time $\tau_n$) we know that $\sum_{v\in[n]\setminus[\tau_n]} \mathds 1\{D_v(n)=k\}$ is independent of $G_{\tau_n}$. Therefore, we can simplify the first term in \eqref{eqn:decomposition} as
\begin{align}
\frac{1}{\sqrt{n}}\left(\E_1\left[\left.N_k(n)\right|G_{\tau_n}\right]-\E_1\left[N_k(n)\right]\right) &= \frac{1}{\sqrt{n}}\left(\E_1\left[\left.N_k(\tau_n,n)\right|G_{\tau_n}\right]-\E_1\left[N_k(\tau_n,n)\right]\right)\ .
\end{align}

At this point, it is useful to write $\E_1\left[\left.N_k(\tau_n,n)\right|G_{\tau_n}\right]$ in a slightly more explicit way. Note that
\begin{align}
\E_1\left[\left.N_k(\tau_n,n)\right|G_{\tau_n}\right] &=\E_1\left[\left. \sum_{v\in[\tau_n]} \mathds 1\{D_v(n)=k\}\right|G_{\tau_n}\right]\\
&=\sum_{v\in[\tau_n]} \sum_{j=m}^k \mathds 1\{D_v(\tau_n)=j\} \underbrace{\P_1\left(\left. D_v(n)=k\right|D_v(\tau_n)=j\right)}_{\coloneqq p_{j,k}(\tau_n,n)}\\
&=\sum_{j=m}^k N_j(\tau_n) p_{j,k}(\tau_n,n)\ .
\end{align}
Using this we immediately see that
$$\E_1[N_k(\tau_n,n)]=\sum_{j=m}^k \E_1[N_j(\tau_n)] p_{j,k}(\tau_n,n)=\sum_{j=m}^k \E_0[N_j(\tau_n)] p_{j,k}(\tau_n,n)\ .$$
In conclusion
\begin{align}
\lefteqn{\frac{1}{\sqrt{n}}\left(\E_1\left[\left.N_k(n)\right|G_{\tau_n}\right]-\E_1\left[N_k(n)\right]\right)}\\
&= \frac{1}{\sqrt{n}} \sum_{j=m}^k \left(N_j(\tau_n) -\E_0[N_j(\tau_n)]\right) p_{j,k}(\tau_n,n)\\
&= \sum_{j=m}^k \frac{1}{\sqrt{\tau_n}}\left(N_j(\tau_n) -\E_0[N_j(\tau_n)]\right) \sqrt{\frac{\tau_n}{n}}p_{j,k}(\tau_n,n)\ .\label{eqn:almost_final}
\end{align}
To proceed all that is needed is to characterize $p_{j,k}(\tau_n,n)$. Lemma~\ref{lem:crucial} immediately provides the necessary result, specifically
$$p_{k,k}(\tau_n,n)=1-(1+\smallO(1))cn^{\gamma-1}m\frac{k+\delta_1}{2m+\delta_1}=1+\smallO(1)$$
and for $j<k$
$$p_{j,k}(\tau_n,n)\leq \P_1\left(\left. D_v(n)>k\right|D_v(\tau_n)=j\right)=\smallO(1)\ .$$
In conclusion, since $\sqrt{\tau_n/n}\to 1$ and $p_{k,l}(\tau_n,n)\to \mathds 1\{j=k\}$ as $n\to\infty$ we conclude that \eqref{eqn:almost_final} converges to the same normal distribution as $\frac{1}{\sqrt{n}}\left(N_k(n) -\E_0[N_k(n)]\right)$ under the null model.

Owing to the linearity of \eqref{eqn:almost_final}, the same argument also shows that any finite linear combination of the (centered and rescaled) elements of $\left(N_m(n),\ldots,N_k(n)\right)$ is asymptotically normal with the appropriate variance. An application of the Cram\'{e}r-Wold device then shows that, for any $k\geq m$, this vector converges in distribution to the desired finite-dimensional multi-variate Gaussian distribution. This concludes the proof.
\end{proof}

With this lemma at hand \eqref{eq:asymptotic_normality_known_delta_test_statistic_alternative_hypothesis} is immediate. The second result \eqref{eq:asymptotic_normality_unknown_delta_test_statistic_alternative_hypothesis} is, however, not a trivial consequence of the lemma, as the convergence in the product topology in Lemma~\ref{lem:joint_CLT} is unfortunately not sufficient to obtain the final result.

To formally show \eqref{eq:asymptotic_normality_unknown_delta_test_statistic_alternative_hypothesis}, begin by recalling that $\iota'_n(\hat\delta_n)=0$. Expanding $\delta\mapsto\iota'_n(\delta)$ around $\delta_0$ we get
\begin{align}
\hat \delta_n-\delta_0&=-\frac{\iota'_n(\delta_0)}{\iota''_n(\bar \delta_n)}\ ,
\end{align}
where $\bar \delta_n=\delta_0+\bar\zeta_n (\hat\delta_n-\delta_0)$ for some $\bar\zeta_n\in[0,1]$. Therefore, and considering also a Taylor expansion of $p_m(\hat\delta_n)$, we conclude that
\begin{align}%
Q(g_n) &= N_m(n)-np_m(\hat\delta_n)\\
&=N_m(n)-np_m(\delta_0)+n\frac{p_m'(\breve \delta_n)}{\iota''_n(\bar \delta_n)}\iota'_n(\delta_0)\ ,
\end{align}%
where $\breve \delta_n=\delta_0+\breve\zeta_n (\hat\delta_n-\delta_0)$ for some $\breve\zeta_n\in[0,1]$. Since $\hat\delta_n \xrightarrow{\P_1} \delta_0$ as $n\to\infty$, and using the uniform convergence results of $\iota''$ from \cite{Gao2017} (as already used in the proof of Proposition~\ref{prp:B}) we conclude that 
\begin{align}%
\frac{p'_m(\check\delta_n)}{{\iota''_n(\bar \delta_n)}} \xrightarrow{\P_1} \frac{p'_m(\delta_0)}{\iota''(\delta_0)}\neq 0.
\end{align}%
To aid the presentation we rewrite $Q(G_n)$ as
\begin{align}%
Q(G_n) &= N_m(n)-np_m(\delta_0)+n\frac{p'_m(\delta_0)}{\iota''(\delta_0)}\iota'_n(\delta_0)\\
&\qquad + n\left(\frac{p'_m(\breve\delta_0)}{\iota''_n(\bar\delta_0)}-\frac{p'_m(\delta_0)}{\iota''(\delta_0)}\right)\iota'_n(\delta_0)\ .\label{eq:Q_asympt_normal_second_term_rewritten}
\end{align}%
We argue that the last term is negligible, and therefore it suffices to characterize the asymptotic normality of the first two terms. Note that $N_m(n)$ and $\iota'_n(\delta_0)$ are not independent. To avoid unnecessarily cluttering the presentation we focus first on the asymptotic normality of $\iota'_n(\delta_0)$ and then argue that extending the analysis to the joint normality is straightforward.

The following result generalizes \cite[Lemma 7]{Gao2017} to the alternative model:

\begin{lemma}\label{lem:iota_prime_normality}
Recall the definition of $v$ in \eqref{v-var-def}. Under the alternative model,
\begin{align*}%
\sqrt{n}(\iota_n'(\delta_0) - \mathbb E_1[\iota_n'(\delta_0)]) \xrightarrow{D} \mathcal N(0, v(\delta_0,m))\ .
\end{align*}%
\end{lemma}%
\begin{proof}%
We proceed similarly as in the proof of Lemma \ref{lem:joint_CLT} by isolating the contributions of the vertices that join after the changepoint as follows
\begin{align}%
\sqrt{n}(\iota_n'(\delta_0) - \E_1[\iota_n'(\delta_0)]) &= \frac{\sqrt{n}}{n+1}\sum_{k\geq m}\frac1{k+\delta_0}(N_{>k}(n) - \E_1[N_{>k}(n)\vert G_{\tau_n}]) \\
&\quad+ \frac{\sqrt{n}}{n+1}\sum_{k\geq m}\frac1{k+\delta_0}(\E_1[ N_{>k}(n)\vert G_{\tau_n}] - \E_1[ N_{>k}(n) ])\ .\label{eq:iota_prime_asymptotically_normal_decomposition}
\end{align}%
Following the exact same argument as in the proof of Lemma~\ref{lem:AZ_delta} the first term on the right-hand side is sufficiently small since
\begin{align}%
\P_1(n\vert \iota'_n(\delta_0) - \E_1[\iota'_n(\delta_0)\vert G_{\tau_n}]\vert \geq x) \leq 2 \exp\left(\frac{x^2}{2(n-\tau_n)c^2_{n,m}}\right)\ ,
\end{align}%
where $c_{n,m} = 2m\log(n)(1+o(1))$, and thus 
\begin{align}%
\sqrt{n}(\iota'_n(\delta_0) - \E_1[\iota'_n(\delta_0)\vert G_{\tau_n}]) = \bigO_{\P}(n^{\gamma-1}\log(n)) = \smallO_{\P_1}(1)\ .
\end{align}%
For the following term, it is convenient to introduce a convenient notation, similarly as done in the proof of Lemma~\ref{lem:joint_CLT}:
\begin{align}%
N_{>k}(\tau_n,n) := \sum_{v\in[\tau_n]}\mathds 1\{D_v(n)>k\}\ .
\end{align}%
Note that
\begin{align}%
\E_1[N_{>k}(n)\vert G_{\tau_n}] - \E_1[N_{>k}(n)] &= \E_1[N_{>k}(\tau_n, n)\vert G_{\tau_n}] - \E_1[N_{>k}(\tau_n, n)]\ ,
\end{align}%
and also
\begin{align}
\E_1[N_{>k}(n)\vert G_{\tau_n}] &= N_{>k}(\tau_n)+\sum_{v\in[\tau_n]} \mathds 1\{D_v(\tau_n)\leq k\}\underbrace{\P_1(D_v(n)>k\vert D_v(\tau_n)\leq k)}_{\coloneqq q_{k}(\tau_n,n)}\ .
\end{align}
With this in hand, the second term on the right-hand side of \eqref{eq:iota_prime_asymptotically_normal_decomposition} is
\begin{align}%
&\frac{\sqrt{n}}{n+1}\sum_{k\geq m}\frac1{k+\delta_0}\left(\E_1[ N_{>k}(\tau_n, n)\vert G_{\tau_n}] - \E_1[ N_{>k}(\tau_n, n) ]\right) \\
&= \frac{\sqrt{n\tau_n}}{n+1}\frac1{\sqrt{\tau_n}}\sum_{k\geq m}\frac1{k+\delta_0}\left( N_{>k}(\tau_n) - \E_1[ N_{>k}(\tau_n) ]\right)\\
&\qquad+\frac{\sqrt{n}}{n+1}\sum_{k\geq m}\frac{q_{k}(\tau_n,n)}{k+\delta_0}\sum_{v\in[\tau_n]} \left(\mathds 1\{D_v(\tau_n)\leq k\} - \P_1(D_v(\tau_n)\leq k)\right)\\
&= \frac{\sqrt{n\tau_n}}{n+1}\frac1{\sqrt{\tau_n}}\sum_{k\geq m}\frac1{k+\delta_0}\left( N_{>k}(\tau_n) - \E_0[ N_{>k}(\tau_n) ]\right) \label{eq:iota_prime_asymptotically_normal_decomposition_main}\\
&\qquad-\frac{\sqrt{n}}{n+1}\sum_{k\geq m}\frac{q_{k}(\tau_n,n)}{k+\delta_0}\sum_{v\in[\tau_n]} \left(\mathds 1\{D_v(\tau_n)> k\} - \P_0(D_v(\tau_n)> k)\right)\ .\label{eq:iota_prime_asymptotically_normal_decomposition_extra}
\end{align}%
The term \eqref{eq:iota_prime_asymptotically_normal_decomposition_main} converges in law to $\mathcal N(0,v(\delta_0,m))$ by \cite[Lemma 7]{Gao2017}. To finalize our argument, we are left to prove that the term \eqref{eq:iota_prime_asymptotically_normal_decomposition_extra} is negligible. We do this by rewriting it as a Doob martingale difference, similarly as was done earlier in the proof. Let
\begin{align}%
M_t \coloneqq \frac{\sqrt{n}}{n+1}\E_1\left[\left.\sum_{k\geq m}\frac{q_{k}(\tau_n,n)}{k+\delta_0} N_{>k}(\tau_n)\right| G_t\right]\ .
\end{align}%
With this notation \eqref{eq:iota_prime_asymptotically_normal_decomposition_extra} is precisely $M_{\tau_n}-M_1$. Furthermore, the martingale differences are bounded by
\begin{align}%
\left| M_t-M_{t-1}\right| &\leq \frac{\sqrt{n}}{n+1}\sum_{k\geq m}\frac{2m}{k+\delta_0}q_k(\tau_n,n)\\
&= \frac{2m\sqrt{n}}{(n+1)\tau_n}\sum_{k\geq m}\sum_{v\in[\tau_n]}\frac{1}{k+\delta_0}\frac{\P_1(D_v(n)>k,D_v(\tau_n)\leq k)}{\P_1(D_v(\tau_n)\leq k)}\\
&\leq \frac{2m\sqrt{n}}{(n+1)\tau_n}\sum_{k\geq m}\sum_{v\in[\tau_n]}\frac{\P_1(D_v(n)>k,D_v(\tau_n)\leq k)}{k+\delta_0}\frac{1}{\P_0(D_v(\tau_n)\leq m)}\\
&\leq \frac{2mn^{-3/2}}{p_m(\delta_0)}(1+\smallO(1))\sum_{k\geq m}\sum_{v\in[\tau_n]}\frac{\P_1(D_v(n)>k,D_v(\tau_n)\leq k)}{k+\delta_0}\ .
\end{align}%
The double-summation is controlled in the proof of Proposition~\ref{prp:ana-diff-means}, and it is the sum of the terms in expressions \eqref{eqn:series1} and \eqref{eqn:series2}. In conclusion
\begin{align}%
\left| M_t-M_{t-1}\right| &\leq \underbrace{\frac{2m^2 c}{(2m+\delta_0)p_m(\delta_0)}}_{\coloneqq \text{const}(m,\delta_0,c)}n^{\gamma-3/2}(1+\smallO(1))\ .
\end{align}%
Based on this and using the Azuma-Hoeffding inequality we conclude that, for any $x\geq 0$
\begin{align}
\P\left(|M_{\tau_n}-M_1|\geq x\right) &\leq 2\exp\left\{-\frac{x^2}{(1+\smallO(1))\tau_n \left(\text{const}(m,\delta_0,c)n^{\gamma-3/2}\right)^2}\right\}\\
&\leq 2\exp\left\{-\frac{x^2}{(1+\smallO(1))\text{const}^2(m,\delta_0,c)n^{2(\gamma-1)}}\right\}\ .
\end{align}
This implies that $M_{\tau_n}-M_1=\bigO_{\P_1}(n^{\gamma-1})=\smallO_{\P_1}(1)$, showing that \eqref{eq:iota_prime_asymptotically_normal_decomposition_extra} is negligible, and concluding the proof of the lemma.
\end{proof}%
Lemma~\ref{lem:iota_prime_normality} suffices to show that the last term in \eqref{eq:Q_asympt_normal_second_term_rewritten} gives a negligible contribution to the limit distribution of $(Q(G_n)-\E_1[Q(G_n)])/\sqrt{n}$, since
\begin{align}
\sqrt{n}\left(\frac{p'_m(\breve\delta_0)}{\iota''_n(\bar\delta_0)}-\frac{p'_m(\delta_0)}{\iota''(\delta_0)}\right)\iota'_n(\delta_0)-\E_1\left[\sqrt{n}\left(\frac{p'_m(\breve\delta_0)}{\iota''_n(\bar\delta_0)}-\frac{p'_m(\delta_0)}{\iota''(\delta_0)}\right)\iota'_n(\delta_0)\right]\\
=\smallO_{\P_1}(1)\bigO_{\P_1}(1)=\smallO_{\P_1}(1)\ .
\end{align}

In conclusion
$$\frac{Q(G_n)-\E_1[G_n]}{\sqrt{n}}=\frac{N_m(n)-\E_1[N_m(n)]}{\sqrt{n}}+\sqrt{n}\frac{p'_m(\delta_0)}{\iota''(\delta_0)}\left(\iota'_n(\delta_0)-\E_1[\iota'_n(\delta_0)]\right)+\smallO_{\P_1}(1)\ .$$
Since $\iota'_n(\delta_0)$ and $N_m(n)$ are not independent we cannot directly rely on Lemma~\ref{lem:iota_prime_normality} and \ref{lem:joint_CLT} to obtain the final result. However, using exactly the same type of argument leads to the following sequence of statements:
\begin{align}
\lefteqn{\frac{Q(G_n)-\E_1[Q(G_n)]}{\sqrt{n}}}\\
&= \smallO_{\P_1}(1)+\frac{\E_1\left[\left.N_m(n)\right|G_{\tau_n}\right]-\E_1\left[N_m(n)\right]}{\sqrt{n}}\\
&\qquad+\frac{\sqrt{n}}{n+1}\sum_{k\geq m} \frac{1}{k+\delta_0}\left(\E_1\left[\left.N_{>k}(n)\right|G_{\tau_n}\right]-\E_1\left[N_{>k}(n)\right]\right)\\
&= \smallO_{\P_1}(1)+\frac{N_m(\tau_n)-\E_0\left[N_m(\tau_n)\right]}{\sqrt{n}}\\
&\qquad+\frac{\sqrt{n}}{n+1}\sum_{k\geq m} \frac{1}{k+\delta_0}\left(N_{>k}(\tau_n)-\E_0\left[N_{>k}(\tau_n)\right]\right)\\
&\xrightarrow{D} \mathcal N(0, w(\delta_0,m)+u(\delta_0,m)),
\end{align}
where the last statement follows from the joint convergence, after appropriate rescaling, of $N_m(\tau_n)$ and $\iota'_{\tau_n}$, which is guaranteed by \eqref{eq:asymptotic_normality_proof_intermediate_result} (cf.~the definition of $\iota'_{n}$ in \eqref{eqn:iota_n_prime}).

\begin{appendix}
\section{Proof of auxiliary results}

\subsection{Proof of Proposition~\ref{prp:ana-diff-means}}
Recall that $\delta\in[\delta_{\text{min}},\delta_{\text{max}}]$. Begin by noting that
\begin{equation}
(n+1)\left(\E_1[\iota'_n(\delta)]-\E_0[\iota'_n(\delta)]\right) = \sum_{k\geq m} \frac{\E_1[N_{>k}(n)]-\E_0[N_{>k}(n)]}{k+\delta}\ .
\end{equation}
The numerator in the above summand can be decomposed in a similar manner as used for the proof of Proposition~\ref{prp:A}:
\begin{align}
\label{eqn:differences}
\lefteqn{\E_1[N_{>k}(n)]-\E_0[N_{>k}(n)]}\\
&= \E_1[N_{>k}(n)]-\E_1[N_{>k}(\tau_n)]+\E_0[N_{>k}(\tau_n)]-\E_0[N_{>k}(n)]\\
&= \left(\E_1[N_{>k}(n)]-\E_1[N_{>k}(\tau_n)]\right)-\left(\E_0[N_{>k}(n)]-\E_0[N_{>k}(\tau_n)]\right)\ .
\end{align}
The first equality holds as the law of $G_{\tau_n}$ is the same under the null and alternative models. Therefore,
\begin{align}
\label{eqn:main_decomposition}
(n+1)\left(\E_1[\iota'_n(\delta)]-\E_0[\iota'_n(\delta)]\right) &= \sum_{k\geq m} \frac{\E_1[N_{>k}(n)]-\E_1[N_{>k}(\tau_n)]}{k+\delta}\\
&\qquad - \sum_{k\geq m} \frac{\E_0[N_{>k}(n)]-\E_0[N_{>k}(\tau_n)]}{k+\delta}\ .
\end{align}
The treatment of the two terms is entirely analogous and it is done simultaneously.

For the rest of the proof, let $\ell\in\{0,1\}$. Clearly
\begin{align*}
\E_\ell[N_{>k}(n)]-\E_\ell[N_{>k}(\tau_n)] &= \sum_{v\in[\tau_n]} \P_\ell(D_v(n)>k)-\P_\ell(D_v(\tau_n)>k)\\
&\qquad + \sum_{v\in[n]\setminus [\tau_n]} \P_\ell(D_v(n)>k)\ .
\end{align*}
Like in the proof of Proposition~\ref{prp:A} we distinguish the behavior of ``old'' vertices (that arrived before the change-point) from the remaining vertices (the ``new'' vertices). The contribution of the latter plays an insignificant role, as we see next. Note that, since $k\geq m$ and for $v\in [n]\setminus [\tau_n]$, the event $D_v(n)>k$ is only possible when there is a vertex $v'>v$ that attached to $v$. Referring to \eqref{eq:attachment_function_alt} we see that the probability of this happening is at most $(m+\delta_\ell)/((2m+\delta_\ell)\tau_n-2m)$. Since there are at most $m(n-\tau_n)$ possible edges that could attach we get the simple bound
$$\sum_{v\in[n]\setminus [\tau_n]} \P_\ell(D_v(n)>k) \leq m(n-\tau_n)^2 \frac{m+\delta_\ell}{(2m+\delta_\ell)\tau_n-2m}=\bigO(n^{2\gamma-1})\ .$$
Using that result and the fact that the largest degree in $G_n$ is at most $nm$, this implies that
    \[
    \sum_{k=m}^{nm} \sum_{v\in[n]\setminus [\tau_n]} \frac{\P_\ell(D_v(n)>k)}{k+\delta}\leq \bigO(n^{2\gamma-1})\sum_{k=m}^{nm}\frac{1}{k+\delta}
    \leq \ \bigO(n^{2\gamma-1}\log n) \ .
    \]
Therefore,
\begin{align}
\lefteqn{\sum_{k\geq m} \frac{\E_\ell[N_{>k}(n)]-\E_\ell[N_{>k}(\tau_n)]}{k+\delta}}\\
&= \sum_{k\geq m} \sum_{v\in[\tau_n]} \frac{\P_\ell(D_v(n)>k)-\P_\ell(D_v(\tau_n)>k)}{k+\delta} \ + \ \bigO(n^{2\gamma-1}\log n)\\
&= \sum_{k\geq m} \sum_{v\in[\tau_n]} \frac{\P_\ell(D_v(\tau_n)\leq k,D_v(n)>k)}{k+\delta} \ + \ \smallO(n^\gamma)\\
&= \sum_{k\geq m} \sum_{v\in[\tau_n]} \frac{\P_\ell(D_v(\tau_n)=k,D_v(n)>k)}{k+\delta}\label{eqn:main_term}\\
&\qquad + \sum_{k> m} \sum_{v\in[\tau_n]} \frac{\P_\ell(D_v(\tau_n)<k,D_v(n)>k)}{k+\delta} \ + \ \smallO(n^\gamma)\label{eqn:extras}\ .
\end{align}

The bulk of the analysis is therefore the characterization of the double summations above. This is somewhat delicate, and requires a good understanding of the behavior of $\P_\ell(D_v(\tau_n)\leq k,D_v(n)>k)$ for $v\in[\tau_n]$ and $k\geq m$. For an arbitrary vertex $v\in[\tau_n]$ and ``small'' $k$ we know that, most likely, the degree of the vertex will not change in $G_n$. Most of the contribution in the above expression will therefore be due to the attachment of a single edge to $v$. This reasoning does not apply when $k$ is ``large'', but in that case the denominator $k+\delta$ is large enough to make the contribution to the above summations negligible.

Note that up to time $\tau_n$ both null model and alternative models coincide. Therefore $\frac1{\tau_n}\sum_{v\in[\tau_n]}\P_\ell(D_v(\tau)=k)= p_k(\delta_0)(1+\smallO(1))$, regardless of the value of $\ell$. To characterize the term \eqref{eqn:main_term} we proceed by truncating that series and using Lemma~\ref{lem:crucial}. Define an auxiliary series $b_n=\lceil n^{1-\gamma}/\log n\rceil$. This is a divergent sequence of integers such that $b_n=\smallO(n^{1-\gamma})$. Then
\begin{align}
\lefteqn{\sum_{k\geq m} \sum_{v\in[\tau_n]} \frac{\P_\ell(D_v(\tau_n)=k,D_v(n)>k)}{k+\delta}}\\
&= \sum_{k\geq m} \sum_{v\in[\tau_n]} \frac{\P_\ell(D_v(n)-D_v(\tau_n)>0\mid D_n(\tau_n)=k)\P_\ell(D_v(\tau_n)=k)}{k+\delta}\\
&= \sum_{k=m}^{b_n} \sum_{v\in[\tau_n]} (1+\smallO(1))cn^{\gamma-1} m \frac{k+\delta_\ell}{2m+\delta_\ell}\frac{\P_\ell(D_v(\tau_n)=k)}{k+\delta}\\
&\qquad + \sum_{k=b_n+1}^\infty \sum_{v\in[\tau_n]} \frac{\P_\ell(D_v(n)-D_v(\tau_n)>0\mid D_n(\tau_n)=k)\P_\ell(D_v(\tau_n)=k)}{k+\delta}\\
&= (1+\smallO(1))cn^{\gamma-1} \frac{m}{2m+\delta_\ell} \tau_n \sum_{k=m}^{b_n} \frac{k+\delta_\ell}{k+\delta}p_k(\delta_0)\label{eqn:series1}\\
&\qquad + \sum_{k=b_n+1}^\infty \sum_{v\in[\tau_n]} \frac{\P_\ell(D_v(n)-D_v(\tau_n)>0\mid D_n(\tau_n)=k)\P_\ell(D_v(\tau_n)=k)}{k+\delta}\label{eqn:series2}\ .
\end{align}

The series in \eqref{eqn:series1} is convergent. For the series in \eqref{eqn:series2}, note that
\begin{align}
\lefteqn{\sum_{k=b_n+1}^\infty \sum_{v\in[\tau_n]} \frac{\P_\ell(D_v(n)-D_v(\tau_n)>0 \mid D_n(\tau_n)=k)\P_\ell(D_v(\tau_n)=k)}{k+\delta}}\\
&\leq \sum_{k=b_n+1}^\infty \sum_{v\in[\tau_n]} \frac{k^2}{b_n^2}\frac{\P_\ell(D_v(\tau_n)=k)}{k+\delta}\\
&= \frac{1}{b_n^2}\sum_{k=b_n+1}^\infty \sum_{v\in[\tau_n]} \E_\ell\left(\frac{D^2_v(\tau_n)}{D_v(\tau_n)+\delta}\mathds 1\{D_v(\tau_n)=k\}\right)\\
&\leq \frac{1}{b_n^2}\sum_{k=1}^\infty \sum_{v\in[\tau_n]} \E_\ell\left(\underbrace{\frac{D_v(\tau_n)}{D_v(\tau_n)+\delta}}_{\leq C_m}D_v(\tau_n)\mathds 1\{D_v(\tau_n)=k\}\right)\\
&=\frac{C_m}{b_n^2} \E_\ell\left(\sum_{v\in[\tau_n]} D_v(\tau_n)\right)\\
&=\frac{C_m}{b_n^2} 2m\tau_n=\bigO(n^{2\gamma-1}\log n)=\smallO(n^{\gamma})\ ,
\end{align}
where $C_m:=\frac{m}{m+\delta_{\min}}$. In conclusion,
\begin{align}
\lefteqn{\sum_{k\geq m} \sum_{v\in[\tau_n]} \frac{\P_\ell(D_v(\tau_n)=k,D_v(n)>k)}{k+\delta}}\\
&= (1+\smallO(1))cn^{\gamma-1} \frac{m}{2m+\delta_\ell} \tau_n \sum_{k\geq m} \frac{k+\delta_\ell}{k+\delta}p_k(\delta_0)\ .\label{eq:D_tau_n_equals_k_D_n_greater_k_asymptotic}
\end{align}


The characterization of \eqref{eqn:extras} is significantly more delicate. Note that
\begin{align}
\lefteqn{\sum_{k>m} \sum_{v\in[\tau_n]} \frac{\P_\ell(D_v(\tau_n)<k,D_v(n)>k)}{k+\delta}}\\
&=\sum_{v\in[\tau_n]} \E_\ell\left[\sum_{k>m} \frac{1}{k+\delta}\mathds 1\{D_v(\tau_n)<k,D_v(n)>k\}\right]\\
&\leq\sum_{v\in[\tau_n]} \E_\ell\left[\sum_{k>m} \frac{1}{D_v(\tau_n)+1+\delta}\mathds 1\{D_v(\tau_n)<k,D_v(n)>k\}\right]\\
&=\sum_{v\in[\tau_n]} \E_\ell\left[\frac{1}{D_v(\tau_n)+1+\delta}\mathds 1\{D_v(n)-D_v(\tau_n)\geq 2\}\sum_{k=D_v(\tau_n)+1}^{D_v(n)-1} 1\right]\\
&=\sum_{v\in[\tau_n]} \E_\ell\left[\mathds 1\{D_v(n)-D_v(\tau_n)\geq 2\}\frac{D_v(n)-D_v(\tau_n)-1}{D_v(\tau_n)+1+\delta}\right]\\
&=\sum_{v\in[\tau_n]} \E_\ell\left[\mathds 1\{D_v(n)-D_v(\tau_n)\geq 1\}\frac{D_v(n)-D_v(\tau_n)-1}{D_v(\tau_n)+1+\delta}\right]\ .\label{eqn:key_formula}
\end{align}
The key quantity to control is the expectation in \eqref{eqn:key_formula}. Let $\sigma_v$ denote the first time after $\tau_n$ when an edge is attached to vertex $v$. With this in hand we can bound \eqref{eqn:key_formula} as
\begin{align}
\lefteqn{\sum_{k>m} \sum_{v\in[\tau_n]} \frac{\P_\ell(D_v(\tau_n)<k,D_v(n)>k)}{k+\delta}}\\
&\leq \sum_{v\in[\tau_n]} \sum_{s\in[n]\setminus[\tau_n]} \E_\ell\left[\mathds 1\{\sigma_v=s\}\frac{D_v(n)-D_v(s)}{D_v(\tau_n)+1+\delta}\right]\label{eqn:A}\ .
\end{align}
The following lemma allows us to bound \eqref{eqn:A}:
\begin{lemma}\label{lem:recursion_expected_degrees}
For $\ell\in \{0,1\}$ and $t\geq \tau>\tau_n$
    \begin{align}
    &\E_\ell\left[D_v(t)+\delta_\ell \mid D_v(\tau)\right]\\
    &\qquad=(D_v(\tau)+\delta_\ell)\prod_{j=1}^{t-\tau}\prod_{i=1}^m \left(1+\frac{1}{2(\tau+j-1)) m + \delta_\ell (\tau+j)+(i-1)}\right)\ .
    \end{align}
    \end{lemma}
We postpone the proof of Lemma \ref{lem:recursion_expected_degrees} to later on. With this in hand, for any $s\in[n]\setminus[\tau_n]$,
\begin{align}
\lefteqn{\E_\ell\left[D_v(n)-D_v(s) \mid D_v(s)\right]}\\
&= \E_\ell\left[D_v(n)+\delta_\ell\mid D_v(s)\right]-(D_v(s)+\delta_\ell)\\
&=(D_v(s)+\delta_\ell)\left(\prod_{j\in[n]\setminus[\tau_n]} \prod_{i=1}^m \left(1+\frac{1}{j(2m+\delta_\ell)-2m+i-1}\right)\ -1\right)\\
&=(1+\smallO(1))(D_v(s)+\delta_\ell)(n-\tau_n)\frac{m}{(2m+\delta_\ell)\tau_n}\\
&\leq \bigO(1)(D_v(s)+\delta_\ell) n^{\gamma-1}\ .
\end{align}
When $s=\sigma_v$ we know that $D_v(s)\leq D_n(\tau_n)+m$ (at time $s$ the first edge was attached to $v$, and therefore at most $m$ edges were attached to $v$ after all the intermediate steps). That means that
\begin{align}
\frac{D_v(s)+\delta_\ell}{D_v(\tau_n)+1+\delta} &\leq \frac{D_v(\tau_n)+m+\delta_\ell}{D_v(\tau_n)+1+\delta}\\
&\leq \frac{D_v(\tau_n)+m+\delta_{\max}}{D_v(\tau_n)+1+\delta_{\min}}\\
&\leq \frac{2m+\delta_{\max}}{m+1+\delta_{\min}} :=\text{const}\ ,
\end{align}
where $\text{const}>0$ is simply a constant. Therefore,
\begin{align}
\lefteqn{\sum_{v\in[\tau_n]} \sum_{s\in[n]\setminus[\tau_n]} \E_\ell\left[\mathds 1\{\sigma_v=s\}\frac{D_v(n)-D_v(s)}{D_v(\tau_n)+1+\delta}\right]}\\
&= \sum_{v\in[\tau_n]} \sum_{s\in[n]\setminus[\tau_n]} \E_\ell\left[ \frac{1}{D_v(\tau_n)+1+\delta} \E_\ell\left[\mathds 1\{\sigma_v=s\}(D_v(n)-D_v(s))\mid D_v(s)\right]\right]\\
&=  \bigO(1)\sum_{v\in[\tau_n]} \sum_{s\in[n]\setminus[\tau_n]} \E_\ell\left[ \frac{1}{D_v(\tau_n)+1+\delta} \E_\ell\left.\left[\mathds 1\{\sigma_v=s\}(D_v(s)+\delta_\ell) n^{\gamma-1}\right|D_v(s)\right]\right]\\
&\leq \bigO(n^{\gamma-1})
\sum_{v\in[\tau_n]} \sum_{s\in[n]\setminus[\tau_n]} \P_\ell\left(\sigma_v=s\right)\\
&= \bigO(n^{\gamma-1}) \sum_{v\in[\tau_n]} \P_\ell\left(D_v(n)-D_v(\tau_n)\geq 1\right)\\
&\leq \bigO(n^{\gamma-1}) \sum_{v\in[\tau_n]} (n-\tau_n)m\E_\ell\left[\frac{D_v(\tau_n)+\delta_\ell}{(2m+\delta_\ell)\tau_n+\delta_\ell}\right]\\
&= \bigO(n^{2\gamma-2}) \E_\ell\left[\sum_{v\in[\tau_n]} D_v(\tau_n)+\delta_\ell\right]= \bigO(n^{2\gamma-1})=\smallO(n^{\gamma})\ ,
\end{align}
where the last inequality follows from the same reasoning used to obtain \eqref{eq:D_tau_n_equals_k_D_n_greater_k_asymptotic}, and the last step follows since $\gamma>\tfrac{1}{2}$. This means that the term in \eqref{eqn:extras} is of smaller order than the term in \eqref{eqn:main_term}. In conclusion,
\begin{align}
\sum_{k\geq m} \frac{\E_\ell[N_{>k}(n)]-\E_\ell[N_{>k}(\tau_n)]}{k+\delta}=(1+\smallO(1))cn^{\gamma} \frac{m}{2m+\delta_\ell} \sum_{k=m}^{\infty} \frac{k+\delta_\ell}{k+\delta}p_k(\delta_0)\ .
\end{align}

We are now ready to go back to \eqref{eqn:main_decomposition} to get 
\begin{align}
\lefteqn{(n+1)\left(\E_1[\iota'_n(\delta)]-\E_0[\iota'_n(\delta)]\right)}\\
&=\sum_{k\geq m} \frac{\E_1[N_{>k}(n)]-\E_1[N_{>k}(\tau_n)]}{k+\delta} - \sum_{k\geq m} \frac{\E_0[N_{>k}(n)]-\E_0[N_{>k}(\tau_n)]}{k+\delta}\\
&=(1+\smallO(1))\frac{cmn^\gamma}{(2m+\delta_0)(2m+\delta_1)}(\delta_1-\delta_0)\sum_{k\geq m}\frac{2m-k}{k+\delta}p_k(\delta_0)\\
&=(1+\smallO(1))\frac{cmn^\gamma}{(2m+\delta_0)(2m+\delta_1)}(\delta_1-\delta_0)\left(-1+\sum_{k\geq m}\frac{2m+\delta}{k+\delta}p_k(\delta_0)\right)\ ,
\end{align}
as required and where in the last step we used the fact that $\sum_{k\geq m} p_k(\delta_0)=1$.
\qed

\subsection{Proof of Lemma~\ref{lem:recursion_expected_degrees}}

Note that between $t$ and $\tau$ the attachment function is affine with parameter $\delta_\ell$. Note also that the graph $G_t$ has precisely $t+1$ vertices and $mt$ edges. Let us describe what happens at each one of the intermediate steps. Let $v$ be a vertex in $G_{\tau}$ and let $D_v(\tau+1,i)$ denote its degree in the graph $G_{\tau+1,i}$, where $i\in\{1,\ldots,m\}$. Then,
\begin{align*}
\lefteqn{\E_\ell[D_v(\tau+1,i)+\delta_0\mid D_v(\tau+1,i-1)]}\\
&= D_v(\tau+1,i-1)+\delta_0+\E_\ell[D_v(\tau+1,i)-D_v(\tau+1,i-1)\mid D_v(\tau+1,i-1)]\\
&= D_v(\tau+1,i-1)+\delta_0+\frac{D_v(\tau+1,i-1)+\delta_0}{2\tau m + \delta_0(\tau+1)+(i-1)}\\
&= (D_v(\tau+1,i-1)+\delta_0)\left(1+\frac{1}{2\tau m + \delta_0(\tau+1)+(i-1)}\right)\ .
\end{align*}
Therefore,
\begin{align*}
\E_{\ell}[D_v(\tau+1)+\delta_0\mid D_v(\tau)] &= \E_\ell[D_v(\tau+1,m)+\delta_0\mid D_v(\tau+1,0)]\\
&=(D_v(\tau)+\delta_0) \prod_{i=1}^m \left(1+\frac{1}{2\tau m + \delta_0(\tau+1)+(i-1)}\right)\ .
\end{align*}
Thus, in general, for $t>\tau$,
\begin{align*}
\lefteqn{\E_\ell[D_v(t)+\delta_0\mid D_v(\tau)]}\\
&= \E_\ell\left[\E_{\ell}[D_v(t)+\delta_0\mid D_v(t-1)] \mid D_v(\tau)\right]\\
&= \E_\ell\left[\E_{\ell}[D_v(t)+\delta_0\mid D_v(t-1)] \mid D_v(\tau)\right]\\
&=\prod_{i=1}^m \left(1+\frac{1}{2(t-1) m + \delta_0 t+(i-1)}\right)\E_\ell\left[D_v(t-1)+\delta_0 \mid D_v(\tau)\right]\\
&\vdots\\
&=(D_v(\tau)+\delta_0)\prod_{j=1}^{t-\tau}\prod_{i=1}^m \left(1+\frac{1}{2(\tau+j-1)) m + \delta_0 (\tau+j)+(i-1)}\right)\ .
\end{align*}
\qed

\subsection{Proof of Proposition~\ref{prp:C2}}

To show consistency of $\hat\delta_n$ note that $\iota'_n(\hat\delta_n)=0$ by definition. Recalling \eqref{eqn:L1_convergence} we conclude immediately that
$$\E[\iota'(\hat\delta_n)]\to 0\ .$$
\cite[Lemma~4]{Gao2017} shows that $\iota'$ has a unique zero at $\delta_0$, and $\iota'(\delta)>0$ for $\delta<\delta_0$ and $\iota'(\delta)<0$ for $\delta>\delta_0$. This immediately implies that $\E_1[|\hat\delta_n-\delta_0|]$ as $n\to\infty$, proving the first assertion in the proposition.

Note also that in Proposition~\ref{prp:B} we have shown that $\tilde\delta_n\to\delta_0$, therefore we also have $\E_\ell[|\hat\delta_n-\tilde\delta_n|]\to 0$. To characterize the rate of convergence of $\hat\delta_n$ to $\tilde\delta_n$ we need the following lemma:
\begin{lemma}\label{lem:AZ_delta} For $\ell\in\{0,1\}$ and $x>0$
$$\P_\ell\left(\sup_{\delta\in[\delta_{\min},\delta_{\max}]}(n+1)\left|\iota'_n(\delta)-\E_\ell[\iota'_n(\delta)]\right|\geq x\right)\leq 2\e^{-\frac{x^2}{2nc^2_{n,m}}}\ ,$$
where $c_{n,m}=\sum_{k=m}^{nm} \frac{2m}{k+\delta_{\min}}\ .$
\end{lemma}
\begin{proof}
To prove this lemma we use a similar argument used in the proof of Theorem~\ref{thm:minimal_degree_test_known_delta}, resorting to the Azuma-Hoeffding's inequality. Note that in the expression of $\iota'_n(\delta)$ in \eqref{eqn:iota_n_prime} only the first term in not deterministic. Begin by constructing the Doob martingale
$$M_t(\delta)=\sum_{k\geq m} \frac{\E_{\ell}\left[N_{>k}(n)\mid G_t\right]}{k+\delta}\ ,$$
where $t\in[n]$. Clearly $M_1(\delta)=\sum_{k\geq m} \frac{\E_\ell\left[N_{>k}(n)\right]}{k+\delta}$ and $M_n(\delta)=\sum_{k\geq m} \frac{N_{>k}(n)}{k+\delta}$, therefore
$$M_n(\delta)-M_1(\delta)=(n+1)\left(\iota'_n(\delta)-\E_\ell[\iota'_n(\delta)]\right)\ .$$
Furthermore, at each timestep in the construction of $G_n$ we add $2m$ edges. Therefore
$$|\E_\ell\left[N_{>k}(n)\mid G_t\right]-\E_\ell\left[N_{>k}(n)\mid G_{t-1}\right]\leq 2m\ ,$$
where $t\in\{2,\ldots,n\}$ and $m\leq k\leq nm$. As a result,
$$|M_t(\delta)-M_{t-1}(\delta)| \leq \sum_{k=m}^{nm} \frac{2m}{k+\delta}\leq \sum_{k=m}^{nm} \frac{2m}{k+\delta_{\min}}\ .$$
Note that the bound on the martingale differences holds uniformly in $\delta$. With this in hand we can simply apply the Azuma-Hoeffding's inequality to get the desired result.
\end{proof}

Note that $c_{n,m}=2m(1+\smallO(1)\log n$ as $n\to\infty$. Let $a_n$ be an arbitrary sequence satisfying $a_n=\omega\left(\sqrt{n}\log n\right)$. The above lemma tells us that
$$\P_\ell\left(\sup_{\delta\in[\delta_{\min},\delta_{\max}]}\ \left|\iota'_n(\delta)-\E_\ell[\iota'_n(\delta)]\right|\geq a_n/n\right)=\smallO(1)\ .$$
Now define $h_\ell:\delta\mapsto \R$ as $h_\ell(\delta)\coloneqq \E_\ell[\iota'_n(\delta)]$. 

We have in particular that
$$\iota'_n(\hat\delta_n)-h_\ell(\hat \delta_n)=\smallOp(a_n/n)\ .$$
Using a Taylor expansion of $h_\ell$ allows us to characterize the difference between $\hat\delta_n$ and $\tilde\delta_n$. Let $h'_\ell(\delta)=\frac{\partial}{\partial \delta} h_\ell(\delta)=\E_\ell[\iota''_n(\delta)]$ and recall that $\iota_n(\hat\delta_n)=h_\ell(\tilde\delta_n)=0$ by definition. Then
\begin{align}
0&= \iota_n(\hat\delta_n)-h_\ell(\tilde\delta_n)\\
&= \iota_n(\hat\delta_n)-h_\ell(\hat\delta_n)-h_\ell'(\bar\delta_n)(\tilde\delta_n-\hat\delta_n)\\
&= \smallOp(a_n/n)-h_\ell'(\bar\delta_n)(\tilde\delta_n-\hat\delta_n)\label{eqn:rate_delta_n}\ ,
\end{align}
where $|\bar\delta_n-\hat\delta_n|\leq|\tilde\delta_n-\hat\delta_n|$. To proceed we must understand the behavior of $h'_\ell(\delta)=\E_\ell[\iota''_n(\delta)]$. As argued in the proof of Proposition~\ref{prp:B}, thanks to the results in \cite{Gao2017} and the fact that $\iota''_n(\delta)$ and $\iota''(\delta)$ are uniformly bounded,
$$\E_0\left[\sup_{\delta\in[\delta_{\min},\delta_{\max}]}\ |\iota''_n(\delta)-\iota''(\delta)|\right]\to 0\ .$$
Actually, this result also holds under the alternative hypothesis by using the following fact:
\begin{lemma}\label{lem:second_derivative_difference}
Let $\tfrac{1}{2}<\gamma<1$. Then $\E_1[\iota''_n(\delta)]-\E_0[\iota''_n(\delta)]\to 0$ as $n\to\infty$ uniformly in $\delta\in[\delta_{\min},\delta_{\max}]$.
\end{lemma}
The proof of this result follows almost immediately from the arguments used to prove Proposition~\ref{prp:ana-diff-means}.
In addition, it is also shown in \cite{Gao2017} that $\iota''(\delta_0)<0$. Since $\bar\delta_n$ converges in probability to $\delta_0$ we therefore conclude that
$$\E_\ell[\iota''_n(\bar\delta_n)]=(1+\smallO(1))\iota''(\delta_0)<0\ .$$
This, together with \eqref{eqn:rate_delta_n} implies that
$$\hat\delta_n-\tilde\delta_n=\smallOp(a_n/n)\ ,$$
proving the second statement in the proposition. The third statement is a rather trivial consequence of the second statement by using a Taylor expansion of $p_m(\delta)$ around $\delta_0$ (recalling Equation~\eqref{eqn:p_m_prime}) to obtain
$$n(p_m(\hat\delta_n)-p_m(\tilde\delta_n))=n p'_m(\delta_0)(1+\smallOp(1))(\hat\delta_n-\tilde\delta_n)=\smallO_{\P_\ell}(a_n)\ .$$
\qed
\subsection{Sketch proof of Lemma~\ref{lem:second_derivative_difference}}
Similarly to the proof of Proposition~\ref{prp:ana-diff-means} note that
\begin{align}
\E_1[\iota''_n(\delta)]-\E_0[\iota''_n(\delta)]&=-\frac{1}{n+1}\sum_{k\geq m} \frac{\E_1[N_{>k}(n)]-\E_0[N_{>k}(n)]}{(k+\delta)^2}\ .
\end{align}
This is quite similar to
\begin{align}
\E_1[\iota'_n(\delta)]-\E_0[\iota'_n(\delta)]=\frac{1}{n+1}\sum_{k\geq m} \frac{\E_1[N_{>k}(n)]-\E_0[N_{>k}(n)]}{k+\delta}\ .
\end{align}
Note that $\sum_{k\geq m} \frac{1}{(k+\delta)^2}$ is a convergent series, unlike $\sum_{k\geq m} \frac{1}{k+\delta}$. By the same (and in fact somewhat simpler) arguments as in the proof of Proposition~\ref{prp:ana-diff-means} we conclude therefore that $\E_1[\iota''_n(\delta)]-\E_0[\iota''_n(\delta)]=\smallO(n^{\gamma-1})$, as we wanted to show.
\qed
\end{appendix}
%
%

%
\begin{funding}
The work of RvdH was supported in part by the Netherlands Organisation for Scientific Research (NWO) through Gravitation-grant {\sc NETWORKS}-024.002.003.
\end{funding}


\DeclareRobustCommand{\HOF}[3]{#3} 
\bibliographystyle{imsart-number_new} 
\bibliography{library}       

\begin{thebibliography}{40}

\bibitem{Adamic2000}
\begin{barticle}[author]
\bauthor{\bsnm{Adamic},~\bfnm{Lada~A.}\binits{L.~A.}},
  \bauthor{\bsnm{Huberman},~\bfnm{Bernardo~A.}\binits{B.~A.}},
  \bauthor{\bsnm{Barab{\'{a}}si},~\bfnm{A.~L.}\binits{A.~L.}},
  \bauthor{\bsnm{Albert},~\bfnm{R.}\binits{R.}},
  \bauthor{\bsnm{Jeong},~\bfnm{H.}\binits{H.}} \AND
  \bauthor{\bsnm{Bianconi},~\bfnm{G.}\binits{G.}}
(\byear{2000}).
\btitle{{Power-law distribution of the world wide web}}.
\bjournal{Science}
\bvolume{287.5461}
\bpages{2115}.
\bdoi{10.1126/science.287.5461.2115a}
\end{barticle}
\endbibitem

\bibitem{Albert2000}
\begin{barticle}[author]
\bauthor{\bsnm{Albert},~\bfnm{R{\'{e}}ka}\binits{R.}} \AND
  \bauthor{\bsnm{Barab{\'{a}}si},~\bfnm{Albert~L{\'{a}}szl{\'{o}}}\binits{A.~L.}}
(\byear{2000}).
\btitle{{Topology of evolving networks: Local events and universality}}.
\bjournal{Physical Review Letters}
\bvolume{85.24}
\bpages{5234--5237}.
\bdoi{10.1103/PhysRevLett.85.5234}
\end{barticle}
\endbibitem

\bibitem{Baldassarri2021}
\begin{barticle}[author]
\bauthor{\bsnm{Baldassarri},~\bfnm{Simone}\binits{S.}} \AND
  \bauthor{\bsnm{Bet},~\bfnm{Gianmarco}\binits{G.}}
(\byear{2022}).
\btitle{Asymptotic Normality of Degree Counts in a General Preferential
  Attachment Model}.
\bjournal{Markov Processes and Related Fields}
\bvolume{28.4}
\bpages{577-603}.
\end{barticle}
\endbibitem

\bibitem{Banerjee2018}
\begin{barticle}[author]
\bauthor{\bsnm{Banerjee},~\bfnm{Sayan}\binits{S.}},
  \bauthor{\bsnm{Bhamidi},~\bfnm{Shankar}\binits{S.}} \AND
  \bauthor{\bsnm{Carmichael},~\bfnm{Iain}\binits{I.}}
(\byear{2023}).
\btitle{{Fluctuation bounds for continuous time branching processes and
  evolution of growing trees with a change point}}.
\bjournal{Annals of Applied Probability}
\bvolume{33.4}
\bpages{2919--2980}.
\bdoi{10.1214/22-AAP1881}
\end{barticle}
\endbibitem

\bibitem{Barabasi1999}
\begin{barticle}[author]
\bauthor{\bsnm{Barab{\'{a}}si},~\bfnm{Albert~L{\'{a}}szl{\'{o}}}\binits{A.~L.}}
  \AND \bauthor{\bsnm{Albert},~\bfnm{R{\'{e}}ka}\binits{R.}}
(\byear{1999}).
\btitle{{Emergence of scaling in random networks}}.
\bjournal{Science}
\bvolume{286.5439}
\bpages{509--512}.
\bdoi{10.1126/science.286.5439.509}
\end{barticle}
\endbibitem

\bibitem{Barabasi2002}
\begin{barticle}[author]
\bauthor{\bsnm{Barab{\'{a}}si},~\bfnm{A.~L.}\binits{A.~L.}},
  \bauthor{\bsnm{Jeong},~\bfnm{H.}\binits{H.}},
  \bauthor{\bsnm{N{\'{e}}da},~\bfnm{Z.}\binits{Z.}},
  \bauthor{\bsnm{Ravasz},~\bfnm{E.}\binits{E.}},
  \bauthor{\bsnm{Schubert},~\bfnm{A.}\binits{A.}} \AND
  \bauthor{\bsnm{Vicsek},~\bfnm{T.}\binits{T.}}
(\byear{2002}).
\btitle{{Evolution of the social network of scientific collaborations}}.
\bjournal{Physica A: Statistical Mechanics and its Applications}
\bvolume{311.3-4}
\bpages{590--614}.
\bdoi{10.1016/S0378-4371(02)00736-7}
\end{barticle}
\endbibitem

\bibitem{Bhamidi2018}
\begin{barticle}[author]
\bauthor{\bsnm{Bhamidi},~\bfnm{Shankar}\binits{S.}},
  \bauthor{\bsnm{Jin},~\bfnm{Jimmy}\binits{J.}} \AND
  \bauthor{\bsnm{Nobel},~\bfnm{Andrew}\binits{A.}}
(\byear{2018}).
\btitle{{Change point detection in network models: Preferential attachment and
  long range dependence}}.
\bjournal{Annals of Applied Probability}
\bvolume{28.1}
\bpages{35--78}.
\bdoi{10.1214/17-AAP1297}
\end{barticle}
\endbibitem

\bibitem{Bhattacharjee2020}
\begin{barticle}[author]
\bauthor{\bsnm{Bhattacharjee},~\bfnm{Monika}\binits{M.}},
  \bauthor{\bsnm{Banerjee},~\bfnm{Moulinath}\binits{M.}} \AND
  \bauthor{\bsnm{Michailidis},~\bfnm{George}\binits{G.}}
(\byear{2020}).
\btitle{Change point estimation in a dynamic stochastic block model}.
\bjournal{The Journal of Machine Learning Research}
\bvolume{21.1}
\bpages{4330--4388}.
\end{barticle}
\endbibitem

\bibitem{Bollobas2004}
\begin{barticle}[author]
\bauthor{\bsnm{Bollob{\'{a}}s},~\bfnm{B{\'{e}}la}\binits{B.}} \AND
  \bauthor{\bsnm{Riordan},~\bfnm{Oliver}\binits{O.}}
(\byear{2004}).
\btitle{{The diameter of a scale-free random graph}}.
\bjournal{Combinatorica}
\bvolume{24.1}
\bpages{5--34}.
\bdoi{10.1007/s00493-004-0002-2}
\end{barticle}
\endbibitem

\bibitem{Bollobas2001a}
\begin{barticle}[author]
\bauthor{\bsnm{Bollob{\'{a}}s},~\bfnm{B{\'{e}}la}\binits{B.}},
  \bauthor{\bsnm{Riordan},~\bfnm{Oliver}\binits{O.}},
  \bauthor{\bsnm{Spencer},~\bfnm{Joel}\binits{J.}} \AND
  \bauthor{\bsnm{Tusn{\'{a}}dy},~\bfnm{G{\'{a}}bor}\binits{G.}}
(\byear{2001}).
\btitle{{The degree sequence of a scale-free random graph process}}.
\bjournal{Random Structures \& Algorithms}
\bvolume{18.3}
\bpages{279--290}.
\bdoi{10.1002/rsa.1009}
\end{barticle}
\endbibitem

\bibitem{Broder2000}
\begin{barticle}[author]
\bauthor{\bsnm{Broder},~\bfnm{Andrei}\binits{A.}},
  \bauthor{\bsnm{Kumar},~\bfnm{Ravi}\binits{R.}},
  \bauthor{\bsnm{Maghoul},~\bfnm{Farzin}\binits{F.}},
  \bauthor{\bsnm{Raghavan},~\bfnm{Prabhakar}\binits{P.}},
  \bauthor{\bsnm{Rajagopalan},~\bfnm{Sridhar}\binits{S.}},
  \bauthor{\bsnm{Stata},~\bfnm{Raymie}\binits{R.}},
  \bauthor{\bsnm{Tomkins},~\bfnm{Andrew}\binits{A.}} \AND
  \bauthor{\bsnm{Wiener},~\bfnm{Janet}\binits{J.}}
(\byear{2000}).
\btitle{{Graph structure in the web}}.
\bjournal{Computer Networks}
\bvolume{33.1}
\bpages{309--320}.
\bdoi{10.1016/S1389-1286(00)00083-9}
\end{barticle}
\endbibitem

\bibitem{Bubeck2017a}
\begin{barticle}[author]
\bauthor{\bsnm{Bubeck},~\bfnm{S{\'{e}}bastien}\binits{S.}},
  \bauthor{\bsnm{Devroye},~\bfnm{Luc}\binits{L.}} \AND
  \bauthor{\bsnm{Lugosi},~\bfnm{G{\'{a}}bor}\binits{G.}}
(\byear{2017}).
\btitle{{Finding Adam in random growing trees}}.
\bjournal{Random Structures \& Algorithms}
\bvolume{50.2}
\bpages{158--172}.
\bdoi{10.1002/rsa.20649}
\end{barticle}
\endbibitem

\bibitem{Bubeck2017}
\begin{barticle}[author]
\bauthor{\bsnm{Bubeck},~\bfnm{S{\'{e}}bastien}\binits{S.}},
  \bauthor{\bsnm{Eldan},~\bfnm{Ronen}\binits{R.}},
  \bauthor{\bsnm{Mossel},~\bfnm{Elchanan}\binits{E.}} \AND
  \bauthor{\bsnm{R{\'{a}}cz},~\bfnm{Mikl{\'{o}}s~Z.}\binits{M.~Z.}}
(\byear{2017}).
\btitle{{From trees to seeds: On the inference of the seed from large trees in
  the uniform attachment model}}.
\bjournal{Bernoulli}
\bvolume{23.4A}
\bpages{2887--2916}.
\bdoi{10.3150/16-BEJ831}
\end{barticle}
\endbibitem

\bibitem{Bubeck2015}
\begin{barticle}[author]
\bauthor{\bsnm{Bubeck},~\bfnm{S{\'{e}}bastien}\binits{S.}},
  \bauthor{\bsnm{Mossel},~\bfnm{Elchanan}\binits{E.}} \AND
  \bauthor{\bsnm{R{\'{a}}cz},~\bfnm{Mikl{\'{o}}s~Z.}\binits{M.~Z.}}
(\byear{2015}).
\btitle{{On the influence of the seed graph in the preferential attachment
  model}}.
\bjournal{IEEE Transactions on Network Science and Engineering}
\bvolume{2.1}
\bpages{30--39}.
\bdoi{10.1109/TNSE.2015.2397592}
\end{barticle}
\endbibitem

\bibitem{Cirkovic2022}
\begin{barticle}[author]
\bauthor{\bsnm{Cirkovic},~\bfnm{Daniel}\binits{D.}},
  \bauthor{\bsnm{Wang},~\bfnm{Tiandong}\binits{T.}} \AND
  \bauthor{\bsnm{Zhang},~\bfnm{Xianyang}\binits{X.}}
(\byear{2022}).
\btitle{Likelihood-based Changepoint Detection in Preferential Attachment
  Networks}.
\arxiv{2206.01076}
\end{barticle}
\endbibitem

\bibitem{crimaldi2005convergence}
\begin{barticle}[author]
\bauthor{\bsnm{Crimaldi},~\bfnm{Irene}\binits{I.}} \AND
  \bauthor{\bsnm{Pratelli},~\bfnm{Luca}\binits{L.}}
(\byear{2005}).
\btitle{Convergence results for multivariate martingales}.
\bjournal{{Stochastic Processes and their Applications}}
\bvolume{115.4}
\bpages{571--577}.
\bdoi{10.1016/j.spa.2004.10.004}
\end{barticle}
\endbibitem

\bibitem{Curien2015}
\begin{barticle}[author]
\bauthor{\bsnm{Curien},~\bfnm{Nicolas}\binits{N.}},
  \bauthor{\bsnm{Duquesne},~\bfnm{Thomas}\binits{T.}},
  \bauthor{\bsnm{Kortchemski},~\bfnm{Igor}\binits{I.}} \AND
  \bauthor{\bsnm{Manolescu},~\bfnm{Ioan}\binits{I.}}
(\byear{2015}).
\btitle{{Scaling limits and influence of the seed graph in preferential
  attachment trees}}.
\bjournal{Journal de l'\'{E}cole polytechnique --- Math{\'{e}}matiques}
\bvolume{2}
\bpages{1--34}.
\bdoi{10.5802/jep.15}
\end{barticle}
\endbibitem

\bibitem{Deijfen2007}
\begin{barticle}[author]
\bauthor{\bsnm{Deijfen},~\bfnm{Maria}\binits{M.}}, \bauthor{\bparticle{van~den}
  \bsnm{Esker},~\bfnm{Henri}\binits{H.}}, \bauthor{\bparticle{van~der}
  \bsnm{Hofstad},~\bfnm{Remco}\binits{R.}} \AND
  \bauthor{\bsnm{Hooghiemstra},~\bfnm{Gerard}\binits{G.}}
(\byear{2007}).
\btitle{{A preferential attachment model with random initial degrees}}.
\bjournal{Arkiv for Matematik}
\bvolume{47.1}
\bpages{41--72}.
\bdoi{10.1007/s11512-007-0067-4}
\end{barticle}
\endbibitem

\bibitem{DomHofHoo10}
\begin{barticle}[author]
\bauthor{\bsnm{Dommers},~\bfnm{S.}\binits{S.}},
  \bauthor{\bsnm{{\HOF{Hofstad}{Van der}{van
  der}}~Hofstad},~\bfnm{Remco}\binits{R.}} \AND
  \bauthor{\bsnm{Hooghiemstra},~\bfnm{G.}\binits{G.}}
(\byear{2010}).
\btitle{Diameters in preferential attachment graphs}.
\bjournal{Journ.\ Stat.\ Phys.}
\bvolume{{\bf 139}}
\bpages{72--107}.
\bdoi{10.1007/s10955-010-9921-z}
\end{barticle}
\endbibitem

\bibitem{Faloutsos1999}
\begin{barticle}[author]
\bauthor{\bsnm{Faloutsos},~\bfnm{Michails}\binits{M.}},
  \bauthor{\bsnm{Faloutsos},~\bfnm{Petras}\binits{P.}} \AND
  \bauthor{\bsnm{Faloutsos},~\bfnm{Christos}\binits{C.}}
(\byear{1999}).
\btitle{{On power-law relationships of the internet topology}}.
\bjournal{Computer Communication Review}
\bvolume{29.4}
\bpages{251--261}.
\bdoi{10.1145/316194.316229}
\end{barticle}
\endbibitem

\bibitem{Farkas2003}
\begin{barticle}[author]
\bauthor{\bsnm{Farkas},~\bfnm{I.}\binits{I.}},
  \bauthor{\bsnm{Jeong},~\bfnm{H.}\binits{H.}},
  \bauthor{\bsnm{Vicsek},~\bfnm{T.}\binits{T.}},
  \bauthor{\bsnm{Barab{\'{a}}si},~\bfnm{A.~L.}\binits{A.~L.}} \AND
  \bauthor{\bsnm{Oltvai},~\bfnm{Z.~N.}\binits{Z.~N.}}
(\byear{2003}).
\btitle{{The topology of the transcription regulatory network in the yeast,
  Saccharomyces cerevisiae}}.
\bjournal{Physica A: Statistical Mechanics and its Applications}
\bvolume{318.3-4}
\bpages{601--612}.
\bdoi{10.1016/S0378-4371(02)01731-4}
\end{barticle}
\endbibitem

\bibitem{Gao2011}
\begin{bphdthesis}[author]
\bauthor{\bsnm{Gao},~\bfnm{Fengnan}\binits{F.}}
(\byear{2011}).
\btitle{{Modeling and interference of the internet movie database}},
\btype{Master Thesis},
\bpublisher{Eindhoven University of Technology}.
\end{bphdthesis}
\endbibitem

\bibitem{Gao2017}
\begin{barticle}[author]
\bauthor{\bsnm{Gao},~\bfnm{Fengnan}\binits{F.}} \AND
  \bauthor{\bparticle{van~der} \bsnm{Vaart},~\bfnm{Aad}\binits{A.}}
(\byear{2017}).
\btitle{{On the asymptotic normality of estimating the affine preferential
  attachment network models with random initial degrees}}.
\bjournal{Stochastic Processes and their Applications}
\bvolume{127.11}
\bpages{3754--3775}.
\bdoi{10.1016/J.SPA.2017.03.008}
\end{barticle}
\endbibitem

\bibitem{Gao2017a}
\begin{barticle}[author]
\bauthor{\bsnm{Gao},~\bfnm{Fengnan}\binits{F.}}, \bauthor{\bparticle{van~der}
  \bsnm{Vaart},~\bfnm{Aad}\binits{A.}},
  \bauthor{\bsnm{Castro},~\bfnm{Rui~M.}\binits{R.~M.}} \AND
  \bauthor{\bparticle{van~der} \bsnm{Hofstad},~\bfnm{Remco}\binits{R.}}
(\byear{2017}).
\btitle{{Consistent estimation in general sublinear preferential attachment
  trees}}.
\bjournal{Electronic Journal of Statistics}
\bvolume{11.2}
\bpages{3979--3999}.
\bdoi{10.1214/17-EJS1356}
\end{barticle}
\endbibitem

\bibitem{VanderHofstad2017}
\begin{bbook}[author]
\bauthor{\bsnm{\HOF{Hofstad}{Van der}{van
  der}~Hofstad},~\bfnm{Remco}\binits{R.}}
(\byear{2017}).
\btitle{{Random graphs and complex networks - Volume one}}.
\bpublisher{Cambridge University Press}.
\bdoi{10.1017/9781316779422}
\end{bbook}
\endbibitem

\bibitem{VanderHofstad2020}
\begin{bbook}[author]
\bauthor{\bsnm{\HOF{Hofstad}{Van der}{van
  der}~Hofstad},~\bfnm{Remco}\binits{R.}}
(\byear{2024}).
\btitle{{Random graphs and complex networks - Volume two}}.
\bpublisher{(in preparation)}.
\end{bbook}
\endbibitem

\bibitem{Jeong2000}
\begin{barticle}[author]
\bauthor{\bsnm{Jeong},~\bfnm{H.}\binits{H.}},
  \bauthor{\bsnm{Tombor},~\bfnm{B.}\binits{B.}},
  \bauthor{\bsnm{Albert},~\bfnm{R.}\binits{R.}},
  \bauthor{\bsnm{Oltval},~\bfnm{Z.~N.}\binits{Z.~N.}} \AND
  \bauthor{\bsnm{Barab{\'{a}}si},~\bfnm{A.~L.}\binits{A.~L.}}
(\byear{2000}).
\btitle{{The large-scale organization of metabolic networks}}.
\bjournal{Nature}
\bvolume{407.6804}
\bpages{651--654}.
\bdoi{10.1038/35036627}
\end{barticle}
\endbibitem

\bibitem{Marchand2020}
\begin{barticle}[author]
\bauthor{\bsnm{Marchand},~\bfnm{David~Corlin}\binits{D.~C.}} \AND
  \bauthor{\bsnm{Manolescu},~\bfnm{Ioan}\binits{I.}}
(\byear{2020}).
\btitle{{Influence of the seed in affine preferential attachment trees}}.
\bjournal{Bernoulli}
\bvolume{26.3}
\bpages{1665--1705}.
\bdoi{10.3150/19-BEJ1152}
\end{barticle}
\endbibitem

\bibitem{Middendorf2005}
\begin{barticle}[author]
\bauthor{\bsnm{Middendorf},~\bfnm{Manuel}\binits{M.}},
  \bauthor{\bsnm{Ziv},~\bfnm{Etay}\binits{E.}} \AND
  \bauthor{\bsnm{Wiggins},~\bfnm{Chris~H.}\binits{C.~H.}}
(\byear{2005}).
\btitle{{Inferring network mechanisms: The Drosophila melanogaster protein
  interaction network}}.
\bjournal{Proceedings of the National Academy of Sciences of the United States
  of America}
\bvolume{102.9}
\bpages{3192--3197}.
\bdoi{10.1073/pnas.0409515102}
\end{barticle}
\endbibitem

\bibitem{Newman2001}
\begin{barticle}[author]
\bauthor{\bsnm{Newman},~\bfnm{M.~E.~J.}\binits{M.~E.~J.}}
(\byear{2001}).
\btitle{{The structure of scientific collaboration networks}}.
\bjournal{Proceedings of the National Academy of Sciences}
\bvolume{98.2}
\bpages{404--409}.
\bdoi{10.1073/pnas.021544898}
\end{barticle}
\endbibitem

\bibitem{Pensky2019}
\begin{barticle}[author]
\bauthor{\bsnm{Pensky},~\bfnm{Marianna}\binits{M.}} \AND
  \bauthor{\bsnm{Zhang},~\bfnm{Teng}\binits{T.}}
(\byear{2019}).
\btitle{{Spectral clustering in the dynamic stochastic block model}}.
\bjournal{Electronic Journal of Statistics}
\bvolume{13.1}
\bpages{678--709}.
\bdoi{10.1214/19-EJS1533}
\end{barticle}
\endbibitem

\bibitem{perk:2014}
\begin{barticle}[author]
\bauthor{\bsnm{Perc},~\bfnm{M.}\binits{M.}}
(\byear{2014}).
\btitle{The Matthew effect in empirical data}.
\bjournal{J. R. Soc. Interface}
\bvolume{11.98}.
\bdoi{10.1098/rsif.2014.0378}
\end{barticle}
\endbibitem

\bibitem{Redner1998}
\begin{barticle}[author]
\bauthor{\bsnm{Redner},~\bfnm{S.}\binits{S.}}
(\byear{1998}).
\btitle{{How popular is your paper? An empirical study of the citation
  distribution}}.
\bjournal{European Physical Journal B}
\bvolume{4.2}
\bpages{131--134}.
\bdoi{10.1007/s100510050359}
\end{barticle}
\endbibitem

\bibitem{Samorodnitsky2016}
\begin{barticle}[author]
\bauthor{\bsnm{Resnick},~\bfnm{Sidney~I.}\binits{S.~I.}} \AND
  \bauthor{\bsnm{Samorodnitsky},~\bfnm{Gennady}\binits{G.}}
(\byear{2016}).
\btitle{{Asymptotic normality of degree counts in a preferential attachment
  model}}.
\bjournal{Advances in Applied Probability}
\bvolume{48.A}
\bpages{283--299}.
\bdoi{10.1017/apr.2016.56}
\end{barticle}
\endbibitem

\bibitem{Wang2018}
\begin{barticle}[author]
\bauthor{\bsnm{Wang},~\bfnm{Daren}\binits{D.}},
  \bauthor{\bsnm{Yu},~\bfnm{Yi}\binits{Y.}} \AND
  \bauthor{\bsnm{Rinaldo},~\bfnm{Alessandro}\binits{A.}}
(\byear{2021}).
\btitle{{Optimal change point detection and localization in sparse dynamic
  networks}}.
\bjournal{The Annals of Statistics}
\bvolume{49.1}
\bpages{203 -- 232}.
\bdoi{10.1214/20-AOS1953}
\end{barticle}
\endbibitem

\bibitem{Wang2014}
\begin{barticle}[author]
\bauthor{\bsnm{Wang},~\bfnm{Heng}\binits{H.}},
  \bauthor{\bsnm{Tang},~\bfnm{Minh}\binits{M.}},
  \bauthor{\bsnm{Park},~\bfnm{Youngser}\binits{Y.}} \AND
  \bauthor{\bsnm{Priebe},~\bfnm{Carey~E.}\binits{C.~E.}}
(\byear{2014}).
\btitle{{Locality statistics for anomaly detection in time series of graphs}}.
\bjournal{IEEE Transactions on Signal Processing}
\bvolume{62.3}
\bpages{703--717}.
\bdoi{10.1109/TSP.2013.2294594}
\end{barticle}
\endbibitem

\bibitem{Wang2008}
\begin{barticle}[author]
\bauthor{\bsnm{Wang},~\bfnm{Mingyang}\binits{M.}},
  \bauthor{\bsnm{Yu},~\bfnm{Guang}\binits{G.}} \AND
  \bauthor{\bsnm{Yu},~\bfnm{Daren}\binits{D.}}
(\byear{2008}).
\btitle{{Measuring the preferential attachment mechanism in citation
  networks}}.
\bjournal{Physica A: Statistical Mechanics and its Applications}
\bvolume{387.18}
\bpages{4692--4698}.
\bdoi{10.1016/j.physa.2008.03.017}
\end{barticle}
\endbibitem

\bibitem{Watts1999}
\begin{bbook}[author]
\bauthor{\bsnm{Watts},~\bfnm{Duncan~J.}\binits{D.~J.}}
(\byear{1999}).
\btitle{{Small Worlds: The Dynamics of Networks between Order and Randomness}}.
\bpublisher{Princeton University Press}.
\end{bbook}
\endbibitem

\bibitem{Watts1998}
\begin{barticle}[author]
\bauthor{\bsnm{Watts},~\bfnm{Duncan~J.}\binits{D.~J.}} \AND
  \bauthor{\bsnm{Strogatz},~\bfnm{Steven~H.}\binits{S.~H.}}
(\byear{1998}).
\btitle{{Collective dynamics of `small-world' networks}}.
\bjournal{Nature}
\bvolume{393}
\bpages{440--442}.
\bdoi{10.1038/30918}
\end{barticle}
\endbibitem

\bibitem{Zhao2019}
\begin{bmisc}[author]
\bauthor{\bsnm{Zhao},~\bfnm{Zifeng}\binits{Z.}},
  \bauthor{\bsnm{Chen},~\bfnm{Li}\binits{L.}} \AND
  \bauthor{\bsnm{Lin},~\bfnm{Lizhen}\binits{L.}}
(\byear{2019}).
\btitle{{Change-point detection in dynamic networks via graphon estimation}}.
\arxiv{1908.01823}
\end{bmisc}
\endbibitem

\end{thebibliography}


\end{document}